\documentclass[12pt,leqno]{amsart}
\usepackage{amsmath,amstext,amsthm,amssymb,amsxtra}
\usepackage{graphicx}
\usepackage{epsfig}
\usepackage{amsfonts}
\usepackage{latexsym}
\usepackage{esint}

\usepackage{comment}

\usepackage[usenames]{color}

\usepackage[left=2.5cm,top=3.1cm,right=2.5cm]{geometry}

\newtheorem{thm}{Theorem}[section]
\newtheorem{prop}[thm]{Proposition}

\newtheorem{cor}[thm]{Corollary}
\newtheorem{lem}[thm]{Lemma}

\theoremstyle{definition}

\newtheorem{rem}[thm]{Remark}
\newtheorem{defn}[thm]{Definition}
\newtheorem{example}[thm]{Example}

\def\vint{\mathop{\mathchoice%
          {\setbox0\hbox{$\displaystyle\intop$}\kern 0.22\wd0%
           \vcenter{\hrule width 0.6\wd0}\kern -0.82\wd0}%
          {\setbox0\hbox{$\textstyle\intop$}\kern 0.2\wd0%
           \vcenter{\hrule width 0.6\wd0}\kern -0.8\wd0}%
          {\setbox0\hbox{$\scriptstyle\intop$}\kern 0.2\wd0%
           \vcenter{\hrule width 0.6\wd0}\kern -0.8\wd0}%
          {\setbox0\hbox{$\scriptscriptstyle\intop$}\kern 0.2\wd0%
           \vcenter{\hrule width 0.6\wd0}\kern -0.8\wd0}}%
          \mathopen{}\int}

\DeclareMathOperator{\diam}{\it d}

\DeclareMathOperator{\dima}{dim_A}
\DeclareMathOperator{\udima}{\overline{dim}_A}
\DeclareMathOperator{\ldima}{\underline{dim}_A}

\DeclareMathOperator{\udimm}{\overline{dim}_M}
\DeclareMathOperator{\ldimm}{\underline{dim}_M}
\DeclareMathOperator{\dimh}{dim_H}

      \newcommand{\N}{{\mathbb N}}
\newcommand{\Z}{{\mathbb Z}}      \newcommand{\R}{{\mathbb R}}

\def\dist{\qopname\relax o{dist}}

\newcommand{\Ha}{{\mathcal H}}
\newcommand{\Mi}{{\mathcal M}}
\newcommand{\W}{{\mathcal W}}

\newcommand{\eps}{\varepsilon}

\title[Hardy--Sobolev inequalities]{In between the inequalities of Sobolev and Hardy}

\author[J.\! Lehrb\"ack]{Juha Lehrb\"ack}   
\address[J.L.]{University of Jyvaskyla, Department of Mathematics and Statistics, P.O. Box 35, FI-40014 University of Jyvaskyla, Finland}
\email{juha.lehrback@jyu.fi}

\author[A.V.\! V\"ah\"akangas]{Antti V. V\"ah\"akangas}
\address[A.V.V.]{University of Jyvaskyla, Department of Mathematics and Statistics, P.O. Box 35, FI-40014 University of Jyvaskyla, Finland\\  and 
University of Helsinki, Department of Mathematics and Statistics,
P.O. Box 68, FI-00014 University of Helsinki, Finland}
\email{antti.vahakangas@iki.fi}

\begin{document}

\keywords{Sobolev inequality, Hardy inequality, Assouad dimension}
\subjclass[2010]{35A23 (26D15, 46E35)}

\begin{abstract}
We establish both sufficient and necessary conditions for 
the validity of the so-called
Hardy--Sobolev inequalities on open sets of the Euclidean space. These inequalities 
form a natural interpolating scale between the (weighted) Sobolev inequalities
and the (weighted) Hardy inequalities. The Assouad dimension
of the complement of the open set  turns out to play an important role in both 
sufficient and necessary conditions.
\end{abstract}
\maketitle

\section{Introduction}

The \emph{Sobolev inequality} is a fundamental tool in all analysis related to
partial differential equations
and variational problems, see e.g.~\cite{MR2597943,Mazya1985}.
When $G\subset\R^n$ is an open set
and $1\le p < n$, this inequality states that 
\begin{equation}\label{eq.sobo}
\biggl(\int_{G} \lvert f\rvert^{np/(n-p)} \,dx\biggr)^{(n-p)/np}
\le  C \biggl(\int_G \lvert \nabla f\rvert^p\,dx\biggr)^{1/p}
\qquad\text{ for all } f\in C_0^\infty(G)\,,
\end{equation}
where the constant $C>0$ depends only on $n$ and $p$.
If $G$ is bounded (or of finite measure)
and $1\le q \le np/(n-p):=p^*$, a simple use of H\"older's
inequality yields a corresponding inequality where on the left-hand side 
of \eqref{eq.sobo}  
the $p^*$-norm is replaced by the $q$-norm; the constant in the inequality
then depends on the measure of $G$ as well.
In particular, for $q=p$ this gives the so-called
\emph{Friedrichs' inequality}
\begin{equation*}\label{eq.friedrich}
\int_{G} \lvert f\rvert^{p} \,dx
\le  C \int_G \lvert \nabla f\rvert^p\,dx
\qquad\text{ for all } f\in C_0^\infty(G)\,.
\end{equation*}
However, if $p>1$ and the open set
$G$ satisfies some additional properties, e.g.\ 
$G$ is a Lipschitz domain or more generally 
the complement of $G$ is uniformly
$p$-fat, then Friedrichs' inequality can
be improved into
a \emph{$p$-Hardy inequality}
\begin{equation}\label{eq.p-hardy}
\int_{G} \lvert f\rvert^{p} \delta_{\partial G}^{-p}\,dx
\le  C \int_G \lvert \nabla f\rvert^p\,dx
\qquad\text{ for all } f\in C_0^\infty(G),
\end{equation}
where $\delta_{\partial G}(x)=\dist(x,\partial G)$ denotes the distance from
$x\in G$ to the boundary of $G$; see e.g.\ Lewis~\cite{Lewis1988}
and Wannebo~\cite{MR1010807}.
Unlike Friedrichs' inequality, this $p$-Hardy inequality can be valid
even if the open set $G$ has infinite measure.
A weighted $(p,\beta)$-Hardy inequality is obtained from inequality~\eqref{eq.p-hardy}
by replacing $dx$ with $\delta_{\partial G}^\beta\,dx$, $\beta\in\R$, on both sides
of~\eqref{eq.p-hardy}. Such an inequality holds, for instance, in a Lipschitz domain $G$
for $1< p < \infty$ if (and only if) and $\beta < p-1$, 
as was shown by Ne\v cas~\cite{MR0163054}.
On the other hand, if, roughly speaking, $\partial G$ contains an isolated
part of dimension $n-p+\beta$, then the $(p,\beta)$-Hardy inequality can 
not be valid in $G\subset\R^n$; we refer to~\cite{MR1948106,MR2442898}.

In this  paper,  we are interested in certain 
inequalities forming a natural 
interpolating scale in between the (weighted) Sobolev inequalities
and  the (weighted) Hardy inequalities. More precisely,
we say that an open set $G\subsetneq\R^n$
{\em admits a $(q,p,\beta)$-Hardy--Sobolev inequality}
if there is a constant $C>0$ such that
the inequality 
\begin{equation}\label{e.hardy-sobo}
\biggl(\int_{G} \lvert f\rvert^q \,\delta_{\partial G}^{(q/p)(n - p + \beta) - n}\,dx\biggr)^{1/q}
\le  C \biggl(\int_G \lvert \nabla f\rvert^p\,\delta_{\partial G}^{\beta} \,dx\biggr)^{1/p}
\end{equation}
holds for all $f\in C^\infty_0(G)$. 
Notice how the Sobolev inequality~\eqref{eq.sobo} is obtained as the case $q=p^*=np/(n-p)$, $\beta=0$ 
in~\eqref{e.hardy-sobo}; and 
the weighted $(p,\beta)$-Hardy inequality is exactly the case $q=p$ in~\eqref{e.hardy-sobo}.

We  begin
in Section~\ref{s.interpolate} by showing that if 
an open set $G\subset\R^n$ admits a
$(p,\beta)$-Hardy inequality, then also the $(q,p,\beta)$-Hardy--Sobolev inequality
holds for all $p\le q\le p^*$, see Theorem~\ref{thm.H to HS}. 
Thus, for these $q$, sufficient conditions for Hardy inequalities always 
yield sufficient conditions for Hardy--Sobolev inequalities, 
see Corollary~\ref{coro.hardy_ass}.
We recall that there are in principle two separate classes of 
open sets 
in which the 
$(p,\beta)$-Hardy inequality can hold: either the complement
 $G^c=\R^n\setminus G$ is
`thick', as in the case of Lipschitz domains or uniformly fat complements,
or then the complement is `thin', corresponding to an upper bound on the 
Assouad dimension $\dima(G^c)$. 
(We refer to Section~\ref{s.dim} for definitions and preliminary results
related to various notions of dimension and to Section~\ref{s.main} for a
more precise formulation of this `dichotomy'
between thick and thin cases.) 
It turns out that actually
in the thick case 
the sufficient conditions emerging from Hardy inequalities
are rather sharp for Hardy--Sobolev inequalities as well; 
see the discussion after Corollary~\ref{coro.hardy_ass} and Theorem~\ref{t.dichotomy}.

On the other hand, the $(p^*,p,0)$-Hardy--Sobolev inequality (that is, 
the Sobolev inequality) holds without any extra assumptions on $G$, and 
this is not the case with the $p$-Hardy inequality.
Hence it is 
natural to expect that
for exponents $p<q<p^*$ 
one could at least in some cases relax the assumptions required for Hardy inequalities
and still obtain the Hardy--Sobolev inequalities. 
We show that this is indeed possible in the case when the complement 
$G^c$ is thin.
More precisely, in~\cite{lehrbackHardyAssouad} it was shown that if
$\dima(G^c)<n-p+\beta$ and $\beta<p-1$, 
then $G$ admits a $(p,\beta)$-Hardy inequality.
Now, in Theorem~\ref{t.hs.thin} 
we show that for the  $(q,p,\beta)$-Hardy--Sobolev inequality,
with $p\le q\le p^*$, it is 
actually  sufficient that
$\dima(G^c)<\min\bigl\{\frac q p (n-p+\beta) , n-1\bigr\}$.
For $q=p$,
this result gives an improvement
for the sufficient condition for the $(p,\beta)$-Hardy inequality as well,
see Remark~\ref{r.impr.ha} for details.

The proof of Theorem~\ref{t.hs.thin} relies heavily on the
work of Horiuchi~\cite{MR1021144}, where the main interest
was in the existence of embeddings between weighted Sobolev spaces. 
Our main contribution to this part is the observation that
the so-called $P(s)$-property (see Definition~\ref{d.horiuchi}) that Horiuchi 
is using as a sufficient condition 
can actually be characterized using the Assouad dimension; this is done in
our Theorem~\ref{t.hori_and_ass}.
It is worth a mention that this adds one more item to the already long list
of notions equivalent to the Assouad dimension, see, for 
instance~\cite{KLV,MR3055588,MR1608518}, and also gives a wealth of 
new examples where Horiuchi's original results can be applied.  
In Section~\ref{s.horiuchi} we present  
Horiuchi's proof, adapted to our
setting, for the sake of clarity and completeness.

It should be noted that in the above results involving a 
`thin' complement  
the test functions do not have to vanish
near $\partial G$, but the inequalities actually hold for all functions in 
$C_0^\infty(\R^n)$. We call such inequalities \emph{global}
Hardy--Sobolev inequalities. In Section~\ref{s.application} 
we establish necessary conditions for these global inequalities,
which  in particular 
yield, together with the sufficient condition from Section~\ref{s.horiuchi},
the following characterization 
in the unweighted case $\beta=0$.

\begin{thm}\label{t.char}
Let $E\not=\emptyset$ be a closed 
 set in $\R^n$ and let $1\le p\le q< np/(n-p)<\infty$. 
Then there is a constant $C>0$ such that the global
$(q,p,0)$-Hardy--Sobolev inequality 
\begin{equation*}\label{e.weighted*}
\biggl(\int_{\R^n} \lvert f\rvert^q \, \delta_E^{(q/p)(n-p)-n}\,dx\biggr)^{1/q}
\le C\biggl(\int_{\R^n} \lvert \nabla f\rvert^p\,dx\biggr)^{1/p}
\end{equation*}
 holds for every $f\in C^\infty_0(\R^n)$
 if and only if
 $\dima(E)<\frac q p (n-p)\,.$
\end{thm}

Sections~\ref{s.horiuchi} and~\ref{s.application} contain, respectively,
both sufficient and necessary conditions for more general global 
Hardy--Sobolev inequalities as well. In the weighted case $\beta\neq 0$
these do not always give full characterizations, but they nevertheless
complement each other and show that
our sufficient conditions are not too far from being optimal. 
For instance, when $\beta<0$,
we need to require in the necessary condition of Theorem~\ref{t.res.beta<0}
that $E$ is either compact (and porous)
or that $E$ satisfies an \emph{a priori} dimensional bound. The
necessity of some extra assumption is shown by an example, whose
justification requires a closer look on the sufficient conditions for 
the Hardy--Sobolev inequality~\eqref{e.hardy-sobo} in the case when the complement 
of the domain $G$ is `thick'
and $\beta\le 0$. Such a result is established in Section~\ref{s.thick}.

Finally, in Section~\ref{s.nec} we prove that an open set which
admits the Hardy--Sobolev inequality~\eqref{e.hardy-sobo} has to satisfy
both local and global dimensional dichotomies: Either the complement
has (locally) a large Hausdorff (or lower Minkowski)
dimension or a small Assouad dimension.
(The global result is formulated earlier in the paper in Theorem~\ref{t.dichotomy}.)
Contrary to the sufficient conditions, as far as we know
no general necessary conditions for
$(q,p,\beta)$-Hardy--Sobolev inequalities have been considered in
the literature when $q>p$. For $q=p$, i.e.\ for $(p,\beta)$-Hardy
inequalities, the corresponding dichotomy is well known, 
see~\cite{MR1948106,MR2442898}. 

We end this introduction with a brief overview of the previously known
sufficient conditions for Hardy--Sobolev inequalities.
In the case when $E\subset\R^n$ is an 
$m$-dimensional subspace, $1\le  m\le n-1$,
and $G=\R^n\setminus E$,
it is due to Maz'ya~\cite[Section 2.1.6]{Mazya1985} that
the global version of the 
$(q,p,\beta)$-Hardy--Sobolev inequality~\eqref{e.hardy-sobo} holds for all functions 
$f\in C_0^\infty(\R^n)$ if $m<\frac q p (n-p+\beta)$;  
notice how
this corresponds to the dimensional bound given above, 
since here $\dima(E)=m$. 
Badiale and Tarantello~\cite{BadialeTarantello2002} 
 (essentially)  rediscovered 
Maz'ya's result for $\beta=0$,
 and applied these inequalities to study the
properties of the solutions for a certain elliptic
partial differential problem related to the 
dynamics of galaxies.
See also Gazzini--Musina~\cite{GazziniMusina2009}
and the references therein for other applications
of Hardy--Sobolev inequalities 
where the distances are taken to
subspaces of $\R^n$.
For  $m=0$, i.e.\ $E=\{0\}$,
the corresponding Hardy--Sobolev inequality is known as
Caffarelli--Kohn--Nirenberg inequality, since
this  case first appeared in~\cite{CaffarelliKohnNirenberg1984}.

For bounded domains with Lipschitz (or H\"older) boundary,
Hardy-Sobolev inequalities have been discussed in \cite[Section~21]{MR1069756}.
Let us also mention that unweighted $(q,p)$-Hardy--Sobolev inequalities follow from the more general
`improved Hardy inequalities' of Filippas, Maz'ya and Tertikas~\cite{FilippasMazyaTertikas2007}, 
under the assumptions that $2\le p < n$, $p< q \le p^*$, $G\subset \R^n$ is a bounded domain
with a $C^2$-smooth boundary, and the distance function satisfies the condition
$-\Delta \delta_{\partial G}\ge 0$ (this 
 is the case $k=1$ 
of~\cite[Theorem~1.1]{FilippasMazyaTertikas2007}).
Moreover, the inequalities in~\cite{FilippasMazyaTertikas2007} contain an additional `Hardy' term 
(with the best constant) on the left-hand side of the inequality~\eqref{e.hardy-sobo},
so these inequalities are much stronger than than the $(q,p)$-Hardy--Sobolev inequality---but 
of course the results in~\cite{FilippasMazyaTertikas2007} are subject to much stronger assumptions, as well. 
For irregular domains satisfying a `plumpness' condition, 
Hardy--Sobolev inequalities have also 
been studied by Edmunds and Hurri-Syrj\"anen  
in~\cite{MR2844457}.

\subsection*{Notation}

Throughout the paper we assume that $G$ is a
non-empty open set in $\R^n$, $n\geq 2$, with
a non-empty boundary. 
The open ball centered at $x\in \R^n$ and with radius $r>0$ is  $B(x,r)$.
The Euclidean
distance from $x\in\R^n$ to a given set $E$ in $\R^n$ is written as $\dist(x,E)=\delta_E(x)$.
The diameter of $E$  is $\diam(E)$.
We write $\chi_E$ for the characteristic function of a set $E$.
The boundary of $E$ is written
as $\partial E$, its closure is written as $\overline{E}$, and 
the complement of $E$ is $E^c=\R^n\setminus E$.
The Lebesgue $n$-measure of a
measurable set $E \subset\R^n$ is
$\lvert E\rvert$. 
If $0<|E|<\infty$,
the integral average of a function $f\in L^1(E)$ is 
$f_E=\vint_E f\,dx = |E|^{-1} \int_E f\,dx$.

All cubes we use are closed and have 
their sides parallel to the coordinate axes.
For a $\lambda>0$ and a cube $Q$ in $\R^n$, we denote by $\lambda Q$ the 
cube with the same center as $Q$ but with side length $\lambda$
times that of $Q$.
The letters $C$ and $c$ 
will denote positive constants whose values are
not necessarily the same at each occurrence. 
If there exists a constant $C>0$ such that $a\le C b$, we sometimes write $a\lesssim b$,
and if $a\lesssim b\lesssim a$ we write $a\simeq b$ and say that $a$ and $b$ are comparable.

\subsection*{Acknowledgments}

J.L.\ wishes to thank Petteri Harjulehto for inspiring 
questions and discussions related to Hardy--Sobolev inequalities.
J. L. has been supported by the Academy of Finland,  grant no.~252108.

\section{Interpolation}\label{s.interpolate}

We show in this section
how (weighted) Hardy--Sobolev inequalities can be
obtained by interpolating between (weighted) Hardy inequalities and (unweighted)
Sobolev inequalities. Recall that the unweighted Sobolev inequalities
are valid for all open sets.

\begin{thm}\label{thm.H to HS}
Assume that $1\le p < n$ and $\beta\in\R$.
 If $G$ admits a $(p,p,\beta)$-Hardy--Sobolev inequality
(i.e., a $(p,\beta)$-Hardy inequality),  then
 $G$ admits $(q,p,\beta)$-Hardy--Sobolev inequalities for all 
exponents $p\le q \le p^*=np/(n-p)$. 
\end{thm}

Let us first prove the following special case; all the other inequalities
can then be obtained with the help of H\"older's inequality.

\begin{lem}\label{lem.w-sobo}
Let $1\le p < n$ and $\beta\in\R$.
 If $G$ admits a $(p,p,\beta)$-Hardy--Sobolev inequality, then
 $G$ admits a $(p^*,p,\beta)$-Hardy--Sobolev inequality. 
\end{lem}

\begin{proof}
Let $f\in C_0^\infty(G)$ and write $g=|f|\delta_{\partial G}^{\beta/p}$.
Then $g$ is a Lipschitz function with a compact support in $G$, and the
gradient of $g$ satisfies (almost everywhere)
\[
|\nabla g|\le |\nabla f|\delta_{\partial G}^{\beta/p} + 
\tfrac{|\beta|}{p}|f|\delta_{\partial G}^{\beta/p-1}\,.
\]
The Sobolev inequality for $g$ (which holds by approximation),
the above estimate for $|\nabla g|$, and the $(p,p,\beta)$-Hardy--Sobolev inequality
(i.e.\ $(p,\beta)$-Hardy inequality) for $f$ imply that
\[\begin{split}
\biggr(\int_G |f&|^{np/(n-p)}  \delta_{\partial G}^{n\beta/(n-p)}\,dx \biggl)^{(n-p)/np}
 = \biggr(\int_G |g|^{np/(n-p)}\,dx \biggl)^{(n-p)/np}
\\&\le C_1 \bigg(\int_G |\nabla g|^p\,dx\bigg)^{1/p}\\
& \le C_1 \bigg\{\bigg( \int_G |\nabla f|^p\delta_{\partial G}^{\beta}\,dx\bigg)^{1/p}
  +  \frac{\lvert \beta\rvert}{p}\bigg(\int_G |f|^p\delta_{\partial G}^{\beta-p}\,dx\bigg)^{1/p}\bigg\}\\
& \le C_2 \bigg( \int_G |\nabla f|^p\delta_{\partial G}^{\beta}\,dx\bigg)^{1/p}\,,
\end{split}\]
which yields the $(p^*,p,\beta)$-Hardy--Sobolev inequality for $f$. 
\end{proof}

\begin{proof}[Proof of Theorem~\ref{thm.H to HS}]
Let $p < q < p^*$ 
and write $\alpha=p^2/(np-nq+qp)$, $\alpha'=\alpha/(\alpha-1)$.
Assume that $f\in C^\infty_0(G)$.
A straightforward computation for the exponents and H\"older's inequality 
(for exponents $\alpha$ and $\alpha'$)
yields 
\begin{equation}\label{e.holder}
\begin{split}
\bigg(\int_{G} \lvert f\rvert^q \,\delta_{\partial G}^{(q/p) (n-p+\beta) - n}\,dx\bigg)^{1/q}
& = \bigg(\int_{G} \lvert f\rvert^{\frac p \alpha + \frac {p^*}{\alpha'}}
             \,\delta_{\partial G}^{\frac{\beta-p}{\alpha}+\frac{n\beta}{(n-p)\alpha'}}\,
             dx\bigg)^{1/q}\\
& \le \biggl(\int_G\lvert f\rvert^p\,\delta_{\partial G}^{\beta-p}\,dx\biggr)^{\frac 1 {q\alpha}}
    \biggl(\int_G\lvert f\rvert^{p^*}\,\delta_{\partial G}^{\frac{n\beta}{n-p}}\,dx\biggr)^{\frac 1 {q\alpha'}}\,.
\end{split}
\end{equation} 
By the assumptions and Lemma~\ref{lem.w-sobo}, we now have available
the two `extreme' Hardy--Sobolev-inequalities, i.e.\ $(p,p,\beta)$- and $(p^*,p,\beta)$-Hardy--Sobolev inequalities.
Using these to the two integrals on the last line of~\eqref{e.holder}, respectively,  
and noting that $\frac 1 {q\alpha}+\frac {n}{n-p}\frac 1 {q\alpha'} = \frac 1 p$,
we obtain from~\eqref{e.holder} 
that
\begin{equation}\label{e.extremes}
\begin{split}
\bigg(\int_{G} \lvert f\rvert^q \,\delta_{\partial G}^{(q/p)(n-p+\beta) - n}\,dx\bigg)^{1/q}
& \le  C_3 \biggl(\int_G\lvert \nabla f\rvert^p\,\delta_{\partial G}^{\beta}\,dx\biggr)^{\frac 1 {q\alpha}}
    \biggl(\int_G \lvert \nabla f\rvert^p\,\delta_{\partial G}^{\beta}\,dx\biggr)^{\frac {n}{n-p}
     \frac 1 {q\alpha'}}\\
& =  C_3 \biggl(\int_G \lvert \nabla f\rvert^p\,\delta_{\partial G}^{\beta} \,dx\biggr)^{1/p}\,,
\end{split}
\end{equation} 
as desired.
\end{proof}

\begin{rem}\label{r.constants}
If $G$ admits a $(q,p,\beta)$-Hardy--Sobolev inequality, 
we use the notation $\kappa_{q,p,\beta}$ for
the best constant appearing in~\eqref{e.hardy-sobo};
recall that $\kappa_{p^*,p,0}<\infty$ for all open sets $G$ in $\R^n$.
In the proof of Lemma~\ref{lem.w-sobo} 
we have $C_1=\kappa_{p^*,p,0}$, and so
we obtain for $\kappa_{p^*,p,\beta}$ the following upper bound:
\[\kappa_{p^*,p,\beta}\le C_2=\kappa_{p^*,p,0}\bigl(1+ \tfrac{\lvert \beta\rvert}{p}\kappa_{p,p,\beta}\bigr)\,.\]
On the other hand, the constant in the proof of Theorem~\ref{thm.H to HS}
is $C_3 = \kappa_{p,p,\beta}^{p/(q\alpha)} \kappa_{p^*,p,\beta}^{p^*/(q\alpha ')}$,
where we have written (as in the proof of Theorem~\ref{thm.H to HS}) 
$\alpha=p^2/(np-nq+qp)$ and $\alpha'=\alpha/(\alpha-1)$.
Thus our interpolation yields the following estimate for the
best constant in the $(q,p,\beta)$-Hardy--Sobolev inequality,
in terms of the constants in the Sobolev and $(p,\beta)$-Hardy inequalities:
\begin{equation*}\label{e.bound*}
\kappa_{q,p,\beta} \le  \kappa_{p,p,\beta}^{p/(q\alpha)} 
\Bigl(\kappa_{p^*,p,0}\bigl(1+ \tfrac{\lvert \beta\rvert}{p}\kappa_{p,p,\beta}\bigr)\Bigr)^{p^*/(q\alpha ')}\,.
\end{equation*}
\end{rem}

\section{Concepts of dimension and the $P(s)$-property}\label{s.dim}

The $\lambda$-dimensional 
Hausdorff measure 
and Hausdorff content
of $E\subset\R^n$ are denoted by $\Ha^\lambda(E)$
and $\Ha_\infty^\lambda(E)$, respectively,
and the Hausdorff dimension of $E$ is
$\dimh(E)$; see \cite[Chapter~4]{MR1333890}.
Besides this well-known notion, we will need several other
concepts of dimension in our results to describe various geometric properties
of sets.

When $A\subset \R^n$ is bounded and $r>0$, we let $N(A,r)$ denote the minimal
number of (open) balls 
of radius $r$ 
and centered at $A$
that are needed to cover the set $A$.  
The $\lambda$-dimensional \emph{Minkowski content} 
of a bounded set $E\subset \R^n$ is then defined to be
\[
\Mi_r^{\lambda}(E)=N(E,r)r^\lambda\,,
\]
and 
the \emph{upper} and \emph{lower Minkowski dimensions} of $E$ 
are
\[
\udimm(E)=\inf\big\{\lambda \ge 0 : \limsup_{r\to 0}\Mi_r^\lambda(E)=0\big\}
\]
and
\[
\ldimm(E)=\inf\big\{\lambda \ge 0 : \liminf_{r\to 0} \Mi_r^\lambda(E)=0\big\}\,,
\]
respectively. 



For general $E\subset\R^n$ we define the following 
`localized' versions of
Minkowski dimensions:
The \emph{(upper) Assouad dimension} of $E$
is defined by setting 
\begin{align*}
&\dima(E) =
\udima(E) \\&= \inf\biggl\{\lambda\ge 0 : N(E\cap B(x,R),r)\le
 C_\lambda\biggl(\frac r R\biggr)^{-\lambda}\ \text{ for all }
 x\in E,\ 0<r<R<\diam(E)\biggr\}\,.
\end{align*}
This upper Assouad dimension is the `usual' Assouad dimension
found in the literature, 
and usually only the notation $\dima(E)$ is used.
Conversely, we define the \emph{lower Assouad dimension} of $E$
to be
\begin{align*}
&\ldima(E) \\&= \sup\biggl\{\lambda\ge 0 : N(E\cap B(x,R),r)\ge
 c_\lambda\biggl(\frac r R\biggr)^{-\lambda}\ \text{ for all }
 x\in E,\ 0<r<R<\diam(E)\biggr\}\,.
\end{align*}
It is clear from the definitions that for a bounded set $E$ in $\R^n$
we always have  
\[
\ldima(E)\le \ldimm(E)\le \udimm(E)\le \udima(E)\,.
\]
In addition, if $E\subset \R^n$ is closed, then
$\ldima(E)\le \dimh(E\cap B)$ for all balls centered in $E$; 
see~\cite[Lemma~2.2]{KLV}. 
We refer to~\cite{Assouad,Fraser,KLV,Larman1967,MR1608518}
for more information on these 
and closely related 
concepts. 

A closed set $E\subset\R^n$ is said to be (Ahlfors) $\lambda$-regular
(or a $\lambda$-set), for $0\le\lambda\le n$, if
there is a constant $C\ge 1$ such that
\[
C^{-1} r^\lambda \le \Ha^{\lambda}(E\cap B(x,r))\le C r^\lambda
\]
for every $x\in E$ and all $0<r<\diam(E)$. If $E$ is a $\lambda$-regular set, then 
all of the above dimensions coincide and are equal to $\lambda$, and so in particular
$\ldima(E) = \udima(E) = \lambda$; see e.g.~\cite{KLV} for details.

Next, we recall the following `Aikawa condition'
for the intergability of the distance function, which is
closely related to the (upper) Assouad
dimension: 
When $\emptyset\not=E\subset \R^n$ is a closed set, we  
let $\mathcal{A}(E)$ be the set of all $s\ge 0$ 
for which there
is a constant $C>0$ such that inequality
\begin{equation}\label{e.assouad}
\int_{B(x,r)} \dist(y,E)^{s-n}\,dy\le C r^s
\end{equation}
holds whenever $x\in E$ and  $0<r< \diam(E)$.
This condition was used by Aikawa~\cite{Aikawa1991} in connection to the so-called
quasiadditivity of capacity,  which has subsequently turned out to be
intimately related to Hardy inequalities, we refer 
to~\cite{lehrbackHardyAssouad,MR2442898,MR3055588}.

The following lemma collects useful properties related to
the Aikawa condition and the (upper) Assouad dimension.
Recall that a set $E \subset \R^n$ is 
said to be \emph{porous}, 
if there is a constant $0<c<1$ such that
for every $x \in E$ and all $0<r<\diam(E)$ there exists a point 
$y\in \R^n$ such that $B(y,cr) \subset B(x,r) \setminus E$.

\begin{lem}\label{l.aikawa}
Let $E\not=\emptyset$ be a closed set. 
\begin{itemize}
\item[(A)]
We have
$\udima(E)=\inf \mathcal{A}(E)\le n$.
\item[(B)] If $\udima(E)<s$ or $n\le s$, 
then $s\in \mathcal{A}(E)$.
\item[(C)] The Aikawa condition is self-improving:
If
$s\in \mathcal{A}(E)$
with $0<s<n$, then there is $0<s'<s$ such that $s'\in\mathcal{A}(E)$;
in particular $\udima(E)<s$. 
\item[(D)] The set $E$ is porous if and only if $\udima(E)<n$.
In particular, we have that $\lvert E \rvert = 0$ if $\udima(E)<n$. 
\item[(E)]
Let $E\not=\emptyset$ be a compact set in $\R^n$ and let $s>0$.
Then $s\in \mathcal{A}(E)$ if and only if for every (or, equivalently, for some) $0<R\le \infty$ there exists
a constant $C>0$  such that
inequality~\eqref{e.assouad} holds whenever $x\in E$ and $0<r<R$.
\end{itemize}
\end{lem}

\begin{proof}
(A) This is proven in~\cite{MR3055588}.

(B) This is easy to see from the definitions and property (A).

(C) 
The proof is based on the Gehring lemma;
see e.g.~\cite[Lemma~2.2]{lehrbackHardyAssouad} for details.

(D) See \cite[Theorem~5.2]{MR1608518}.

(E) Let us outline the proof; the reader will find it straightforward to fill in the details.
We fix a number $0<R<2\diam(E)$ and
assume
that~\eqref{e.assouad} holds whenever $x\in E$ and $0<r<R$.
It suffices to prove that
there exists a constant $C>0$ such that
inequality~\eqref{e.assouad} holds 
 whenever $x\in E$ and $0<r<\infty$;
clearly, we may also assume that $0<s<n$. 
If $0< r< R$, inequality~\eqref{e.assouad} holds by the assumption. 
If $R \le  r\le 2\diam(E)$, 
we use 
the compactness of $E$ 
 to find points $x_1,\ldots,x_K\in E$ such that
\begin{equation}\label{e.rel}
\{ y\in \R^n\,:\, \dist(y,E) < R/4\}  \subset \bigcup_{j=1}^K B(x_j,R/2)\,.
\end{equation}
To estimate the left-hand side of~\eqref{e.assouad}, 
we split the ball $B(x,r)$ in two parts: $A=B(x,r)\cap \{y\,:\,\dist(y,E)<R/4\}$
and  $B(x,r)\setminus A$. 
The integral over the set $A$ is estimated by using~\eqref{e.rel}
and the validity of inequality~\eqref{e.assouad} for all radii up to $R$, 
and the integral over the set $B(x,r)\setminus A$ is easy to
estimate since therein the distances to $E$ are larger than $R/4$ and $r$ is dominated by $2\diam(E)$.
Finally, if
$2\diam(E)<r<\infty$, we split $B(x,r)$ as $D=B(x,2\diam(E))$ and
$B(x,r)\setminus D$. 
The integral over the set $D$ is treated as in the case $R \le  r\le 2\diam(E)$ 
above, and 
the integral over the set $B(x,r)\setminus D$
is estimated by using the fact that therein the distance of 
a point to the set $E$ is comparable to its distance to the point $x$.
 We conclude that~\eqref{e.assouad} holds for all $0<r<\infty$.
\end{proof}

In~\cite{MR1021144}, Horiuchi introduced the following 
$P(s)$-property 
in order to study imbeddings for weighted Sobolev spaces. 
This property was subsequently applied also in~\cite{MR1118940}. 
Here we denote $E_\eta=\{x\in\R^n \,: \delta_E(x)<\eta\}$,
that is, $E_\eta$ is the (open) $\eta$-neighborhood of $E$. 

\begin{defn}\label{d.horiuchi}
 Let $0\le s\le n$. A closed set $E\subset\R^n$ has the property $P(s)$ if $|E|=0$ and
 there is a constant $C>0$ such that
 \[
 |B\cap (E_{\eta_2}\setminus E_{\eta_1})|\le C \eta_2^{s-1}(\eta_2-\eta_1)\diam(B)^{n-s}
 \quad\text{ if } 1\le s \le n 
 \]
and
 \[
 |B\cap (E_{\eta_2}\setminus E_{\eta_1})|\le C (\eta_2-\eta_1)^s \diam(B)^{n-s}
 \quad\text{ if } 0 \le  s < 1\,,
 \]
for all balls $B$ and numbers  $\eta_1,\eta_2$ satisfying $0\le\eta_1<\eta_2\le \diam(B)$. 
\end{defn}

\begin{rem} 
Horiuchi required the (respective) inequality in Definition \ref{d.horiuchi} to hold only for  balls $B$ and numbers $\eta_1,\eta_2$ satisfying inequalities $0\le\eta_1<\eta_2\le\diam(B)\le A_0$ with a fixed $A_0\in (0,\infty]$. For our purposes the given formulation is more suitable. 
Horiuchi also excludes the case $s=0$ in Definition \ref{d.horiuchi}; observe 
that all closed sets with zero Lebesgue measure have the property $P(0)$.  
\end{rem}

We now have the following theorem, 
which characterizes 
the (upper) Assouad dimension in terms of the  $P(s)$-property. 
This result also
clarifies the $P(s)$-property and immediately gives numerous examples of sets 
having this property and thus satisfying the main assumption in
Horiuchi's papers~\cite{MR1021144,MR1118940}.

\begin{thm}\label{t.hori_and_ass}
Let $E\subset\R^n$ be a closed set with $|E|=0$. Then
 \[
 \udima(E) = n - \sup\big\{0\le s\le n \,:\, E 
 \text{ {\rm has the property} } P(s)\big\}.
 \]
 In particular, the $P(s)$-property holds for all
 $0\le s < n-\udima(E)$.
\end{thm}

\begin{proof}
 Assume first that $E$  has the property $P(s)$. 
 Fix a ball $B$ so that $\diam(B)\le \diam(E)$.
 Choosing $\eta_1=0$ and writing $\eta=\eta_2$,
 we find that 
 \[
 |B\cap E_\eta|\le C \eta^{s} \diam(B)^{n-s}
 \]
 for all numbers $0<\eta \le \diam(B)$, 
 regardless of whether $s\ge 1$ or $s<1$.
 From this it follows, by~\cite[Theorem~A.12]{MR1608518}, that
 (in the language of~\cite{MR1608518}) the set $E$ is $(n-s)$-homogeneous,
 and thus $\udima(E)\le n-s$ (see also~\cite[Theorem~5.1]{MR3055588} 
 whose proof is a short argument based on the Aikawa condition).
 This proves one direction (`$\le$') of the claimed equality.

 The converse inequality is somewhat more involved. Since $E$ always has the property $P(0)$, 
 we may  assume that $\udima(E)<n$. Let us fix $0 <s <  n$
 such that $\udima(E)<n-s$. It suffices to show that then $E$ has the property 
 $P(s)$, so let $B=B(w,R)$ be a ball in $\R^n$. 
 Without loss of generality, we may assume that $w\in E$. Let  us fix
 $0\le\eta_1< \eta_2\le \diam(B)=2R$.  If $\eta_1<\eta_2/2$, then 
 $\eta_2-\eta_1\simeq\eta_2$, and the desired estimate 
 follows from Lemma \ref{l.aikawa}(B,E). Indeed, 
 \[
 \lvert B\cap (E_{\eta_2}\setminus E_{\eta_1})\rvert
 \le  \lvert B\cap E_{\eta_2}\rvert\le \eta_2^{s} \int_B \dist(y,E)^{-s}\,dy 
 \lesssim \eta_2^s\, R^{n-s}\,.
\] 
 (Alternatively, this estimate could be obtained from~\cite[Theorem~A.12]{MR1608518} 
 or the proof of~\cite[Theorem~5.1]{MR3055588}).
 
 We may hence assume that $\eta_2/2\le \eta_1<\eta_2$. 
 Let $\W(E^c)=\{B_i\}$ be a Whitney-type cover of $E^c$ with balls
 $B_i=B\bigl(x_i,\frac 1 8 \delta_E(x_i)\bigr)$. In particular, the overlap of these balls is
 uniformly bounded. We also write 
 \[
  \W(E^c; 2B;\eta_1)=\{B_i\in \W(E^c) : 
        B_i\cap 2B \cap \partial E_{\eta_1} \neq \emptyset\}.
 \]
 It follows that $\dist(B_i,E)\simeq \diam(B_i)\simeq \eta_1$ for all $B_i\in\W(E^c;2B;\eta_1)$.
 Thus,
 the assumption $\udima(E)<n-s$ and Lemma \ref{l.aikawa}(B,E) yield
 \begin{equation}\label{e.w-balls}
 \begin{split}
  \#\W(E^c;2B;\eta_1)&\lesssim  \eta_1^{s-n} \sum_{B_i\in \W(E^c;2B;\eta_1)}
  \lvert B_i\rvert^{1-s/n}\\
  &\lesssim \eta_1^{s-n} \int_{cB} \dist(y,E)^{-s}\,dy
 \lesssim \eta_1^{s-n} R^{n-s}\,;
  \end{split}
 \end{equation}
in particular, we obtain that $\#\W(E^c;2B;\eta_1)\lesssim \eta_1^{s-n} R^{n-s}$.  
(See also~\cite[Lemma~4.4]{KLV} for another proof of this estimate.)
 
 In addition, the proof of~\cite[Lemma~5.4]{KLV} shows that, for each ball 
 $B_i\in \W(E^c;2B;\eta_1)$, there exists
 a $2$-Lipschitz mapping from a subset of $\partial B_i$ \emph{onto} 
 $B_i\cap \partial E_{\eta_1}$, and thus
 \begin{equation}\label{e.mink-balls}
 \Mi_\eps^{n-1}(B_i\cap \partial E_{\eta_1}) \le 4^{n-1} 
   \Mi_{C\eps}^{n-1}(\partial B_i) \simeq\eta_1^{n-1},
 \end{equation}
 where we write $\eps = \eta_2-\eta_1$.
 On the other hand, for each $B_i\in\W(E^c;2B;\eta_1)$, let
 $\{ B_j^i\}_j$ be a cover of $B_i\cap\partial E_{\eta_1}$ with balls
 $ B_j^i=B(y_j^i,\eps)$, where $y_j^i\in B_i\cap \partial E_{\eta_1}$
are chosen so that 
 \begin{equation}\label{e.choice-balls}
  \Mi_\eps^{n-1}(B_i\cap\partial E_{\eta_1}) \ge C_n \eps^{-1} \sum_j 
    | B_j^i|. 
 \end{equation}
 Then 
 $B\cap (E_{\eta_2}\setminus E_{\eta_1})\subset\bigcup_{i} \bigcup_{j} 2 B_j^i$, 
 and thus it follows from~\eqref{e.choice-balls}, 
 \eqref{e.mink-balls}, and~\eqref{e.w-balls} that
 \[
 \begin{split}
 |B\cap (E_{\eta_2}\setminus E_{\eta_1})| &  \le
   \sum_{B_i\in \W(E^c;2B;\eta_1)}  \sum_j    
        |2 B_j^i|
   \lesssim \eps \sum_{B_i\in \W(E^c;2B;\eta_1)} \eps^{-1} 
        \sum_{j} 
         | B_j^i|\\
   & \lesssim \eps \sum_{B_i\in \W(E^c;2B;\eta_1)} \Mi_\eps^{n-1}(B_i\cap\partial E_{\eta_1})
   \lesssim \eps \sum_{B_i\in \W(E^c;2B;\eta_1)} \eta_1^{n-1}\\
   & \lesssim \eps \eta_1^{n-1} \eta_1^{s-n} R^{n-s} \le \eta_2^{s-1}(\eta_2-\eta_1) R^{n-s}. 
 \end{split}
 \]
 This shows that $E$ has property $P(s)$, provided that $1\le s <n$. 
 Moreover, if 
 $0<s<1$, then (under the assumption $\eta_2/2\le \eta_1<\eta_2$)
 $\eta_1^{s-1}\le C \eta_2^{s-1} \le C (\eta_2-\eta_1)^{s-1}$, and thus
 \[
 |B\cap (E_{\eta_2}\setminus E_{\eta_1})| \lesssim \eta_1^{s-1}(\eta_2-\eta_1) R^{n-s}
                                     \lesssim (\eta_2-\eta_1)^{s} R^{n-s},
 \]
 proving that $E$ has property $P(s)$ in this case as well.
\end{proof}

\begin{rem}
 In particular, (the proof of) Theorem~\ref{t.hori_and_ass} shows that
 for our purposes we could equivalently define the property $P(s)$ 
 in the case $0<s<1$ in the same way as in the case $1\le s<n$.
\end{rem}

We also record the following lemma that 
will be useful in the proofs of Section~\ref{s.horiuchi}.

\begin{lem}\label{l.equiv}
Let $E\subset\R^n$ be a closed porous set.
If $s>\udima(E)$, 
then
\begin{equation}\label{eq.equiv}
\int_B \delta_E(x)^{s-n}\,dx \simeq \diam(B)^n \bigl(\diam(B) + \dist(B,E)\bigr)^{s-n}
\end{equation}
for all balls $B$ in $\R^n$,
where the constants of comparison are independent of the ball $B$.
\end{lem}

\begin{proof}
This follows from 
Theorem~\ref{t.hori_and_ass} and~\cite[Proposition~6.1]{MR1021144}. 
However, it is 
also easy to give a direct proof
using the Aikawa condition and the porosity of $E$.
\end{proof}

\section{Main results}\label{s.main}

By Theorem~\ref{thm.H to HS}, sufficient conditions for a $(p,\beta)$-Hardy
inequality are also sufficient for $(q,p,\beta)$-Hardy--Sobolev inequalities,
for all $p \le q \le p^*$. Hence we obtain the following corollary
by combining the previously known sufficient conditions for $(p,\beta)$-Hardy inequalities---more 
precisely, Corollary~1.3 in~\cite{lehrbackHardyAssouad}---and Theorem~\ref{thm.H to HS}.

\begin{cor}\label{coro.hardy_ass}
 Let $1 <p\le q\le np/(n-p)<\infty$ and $\beta<p-1$.
 If $G\subset\R^n$
 is an open set and
 \[
 \udima(G^c) < n - p + \beta \quad\text{ or }\quad 
 \ldima(G^c) > n - p + \beta\,,
 \]
 then $G$ admits a $(q,p,\beta)$-Hardy--Sobolev inequality; 
 in the latter case, if $G$ is unbounded, then we require that also $G^c$ is unbounded.
\end{cor}

The second sufficient condition in Corollary \ref{coro.hardy_ass}
for the $(q,p,\beta)$-Hardy--Sobolev inequality in terms of the lower Assouad 
dimension $\ldima(G^c)$,
which corresponds to the case where the complement is `thick', 
turns out to be  rather sharp;
we refer to Theorem~\ref{t.dichotomy} below and also
to Section~\ref{s.nec}.  
Let us remark that, for $p-\beta>1$, the condition 
$\ldima(G^c) > n - p + \beta$ is equivalent to 
$G^c$ being \emph{uniformly $(p-\beta)$-fat}
(see~\cite[Remark~3.2]{KLV}). For the unweighted $p$-Hardy inequality 
(i.e., the $(p,p,0)$-Hardy--Sobolev inequality), 
uniform $p$-fatness is a well-known sufficient condition, see e.g.~\cite{Lewis1988}. 
The bound $\beta<p-1$ in Corollary~\ref{coro.hardy_ass} 
is sharp and necessary for this generality, 
 but  it can be removed for instance under additional accessibility 
conditions for $\partial G$ (cf.~\cite{KoskelaLehrback2009,MR2442898}).

On the other hand,
the 
following 
 Theorem~\ref{t.hs.thin} shows that the first
sufficient condition in Corollary \ref{coro.hardy_ass}, given in terms of the (upper) Assouad dimension
$\udima(G^c)$ and corresponding to the `thin' case, can be weakened
in two ways: The factor $q/p \ge 1$ appears, and the upper bound 
$\beta<p-1$
can be changed to the weaker assumption that $\udima(G^c)<n-1$
(cf.~Remark~\ref{r.impr.ha}).

\begin{thm}\label{t.hs.thin}
Let $1 \le  p\le q\le np/(n-p)<\infty$
and $\beta\in\R$. If $G\subset \R^n$ is an open set and
\begin{equation}\label{e.hs.thin}
\udima(G^c) < \min \biggl\{ \frac{q}{p} (n-p+\beta) \, , \,  n-1  \biggr\}\,,
\end{equation}
then $G$ admits a $(q,p,\beta)$-Hardy--Sobolev inequality.
\end{thm}

Theorem~\ref{t.hs.thin} is a 
consequence of a more general result for functions $f\in C_0^\infty(\R^n)$
that we will establish in Theorem~\ref{t.hardy}.

\begin{rem}\label{r.impr.ha} 
In the case $q=p$, Theorem~\ref{t.hs.thin} in particular improves the sufficient condition
for the $(p,\beta)$-Hardy inequality from~\cite[Corollary~1.3]{lehrbackHardyAssouad}.
In both of these results it is assumed that $\udima(G^c) < n-p+\beta$, but
in~\cite[Corollary~1.3]{lehrbackHardyAssouad} the extra assumption is $\beta<p-1$
instead of $\udima(G^c) < n-1$ in Theorem~\ref{t.hs.thin},
and clearly $\udima(G^c) < n-p+\beta$ and $\beta<p-1$ together imply
that $\udima(G^c) < n-1$.
\end{rem}

Let us give an easy example which shows that the assumption 
$\udima(G^c) < n-1$ in Theorem~\ref{t.hs.thin} can not be removed:

\begin{example}
Let $G=\R^n\setminus \partial B(0,1)$, and consider functions
$f_j\in C_0^\infty(G)$ such that $f_j(x)=1$ when $|x|\le  1-2^{1-j}$,
$f_j(x)=0$ when $|x|\ge 1- 2^{-j}$, and $|\nabla f_j|\le C 2^j$
when $1-2^{1-j}<|x|<1-2^{-j}$. Then, for any $1 \le  p\le q\le np/(n-p)<\infty$
and $\beta\in\R$, the left-hand side of the $(q,p,\beta)$-Hardy--Sobolev 
inequality~\eqref{e.hardy-sobo}
is uniformly bounded away from zero 
for these functions $f_j$ if $j>1$ but, on the other hand,
when $\beta>p-1$ we have for the right hand side of~\eqref{e.hardy-sobo}
that 
\[
\int_G \lvert \nabla f_j\rvert^p\,\delta_{\partial G}^{\beta} \,dx
\lesssim \bigl|\{x\in\R^n\, :\, 1-2^{1-j}<|x|<1-2^{-j}\}\bigr| 2^{jp} 2^{-j\beta} 
\lesssim 2^{-j(\beta-p+1)}\xrightarrow{j\to\infty} 0.
\] 
Hence the  $(q,p,\beta)$-Hardy--Sobolev 
inequality fails in $G$ when $\beta>p-1=p - n + \udima(G^c)$,
even though in this case $\udima(G^c)<n-p+\beta\le\frac q p(n-p+\beta)$,
which is exactly the first bound given by~\eqref{e.hs.thin}.
\end{example}

See also~\cite[Sect.~6.3]{MR2442898}
for a  another example where the open set $G$ is connected.
The calculation is given there only for $p=q$, but the example works, just as above,
for all $1 \le  p\le q\le np/(n-p)<\infty$ and $\beta>p-1$. 
Nevertheless, some
positive results can also be given 
for the case $\udima(G^c)\ge n-1$,
see Theorem~\ref{t.hardy} for details, but these always require an
upper bound for $\beta$.

For the complements of $\lambda$-regular sets (for $0<\lambda<n-1$) the
`thick' and `thin' cases overlap when $p<q$, and hence
all Hardy--Sobolev inequalities hold in this case:

\begin{cor}\label{coro.regular}
 Let $1 < p < q\le np/(n-p)<\infty$ and
 $0<\lambda<n-1$, and assume that $E\subset\R^n$ is an unbounded
 $\lambda$-regular set. Then the open set $G=\R^n\setminus E$
 admits $(q,p,\beta)$-Hardy--Sobolev inequalities for every $\beta\in\R$. 
\end{cor}

\begin{proof}
 For $\beta<p-n+\lambda$ this follows from the `thick' case of
 Corollary~\ref{coro.hardy_ass} 
(observe that $G^c=E$ is assumed to be unbounded), and for $\beta> p-n+\frac p q\lambda$
 from Theorem~\ref{t.hs.thin}. Since $p-n+\frac p q\lambda<p-n+\lambda$,
 these two cases cover all $\beta\in\R$.
\end{proof}

Note that Corollary \ref{coro.regular} is not true when $p=q$, since then the
$(p,p,p-n+\lambda)$-Hardy--Sobolev inequality fails for all $1<p<\infty$ 
by~\cite[Thm~1.1]{MR2442898}.
Contrary to Corollary~\ref{coro.regular}, as soon as the Hausdorff 
(or lower Minkowski) and (upper) Assouad dimensions of the
complement $G^c$ differ, some Hardy--Sobolev inequalities fail in
$G$ by the following Theorem~\ref{t.dichotomy},
which gives a `dichotomy' condition for 
the dimension of the complement $G^c$ when $G$ admits a
$(q,p,\beta)$-Hardy--Sobolev inequality. 
Local versions of these statements, as well
as the proof of Theorem~\ref{t.dichotomy},
are given in Section~\ref{s.nec}.

\begin{thm}\label{t.dichotomy}
Let $G\subset\R^n$ be an open set and
assume that $1\le p\le q<np/(n-p)<\infty$ and $\beta\in\R$ 
are such that $G$ admits a $(q,p,\beta)$-Hardy--Sobolev 
inequality~\eqref{e.hardy-sobo}. 
 
 If $\beta \ge 0$ and $q(n-p+\beta)/p\not=n$,  then either 
 \[\udima(G^c)< \frac{q}{p}(n-p+\beta)\quad \text{or}\quad 
\dimh(G^c)\ge n-p+\beta \,.\] 

 If $\beta<0$ and $G^c$
 is compact and porous, then either 
 \[\udima(G^c)< \frac{q}{p}(n-p+\beta)\quad \text{or}\quad 
\ldimm(G^c)\ge n-p+\beta\,.\] 
\end{thm}

In particular, Theorem~\ref{t.dichotomy} shows that the numbers
$\frac{q}{p}(n-p+\beta)$ and $n-p+\beta$, which bound the dimension of $G^c$ from
above or below in the sufficient conditions in Theorem~\ref{t.hs.thin}
and in the thick case of Corollary~\ref{coro.hardy_ass}, respectively, are sharp, although the notions
of dimension are not  the same
in the lower bounds. 
Recall, however, that when $E\subset\R^n$ is a closed set, 
then $\ldima(E)\le \dimh(E\cap B)$ for all balls $B$ centered at $E$, 
see~\cite[Lemma 2.2]{KLV}.
Moreover, for many sets
it holds that
$\ldima(E)=\dimh(E)$ 
(and even $\ldima(E)=\ldimm(E)$),
so in this sense the sufficient condition
$\ldima(G^c) > n-p+\beta$ in
Corollary~\ref{coro.hardy_ass} is not that far from being
optimal.

Let us also illustrate the applicability and sharpness of our results 
with the following example:

\begin{example}\label{ex.basic}
Consider a closed unbounded set
$E\subset\R^n$ with
\begin{equation}\label{e.ex_main}
0\le\ldima(E)=\dimh(E)=\lambda_1<\lambda_2=\udima(E)<n-1.
\end{equation} 
Then for $\beta\ge 0$
and  $1< p \le q<np/(n-p)<\infty$, with
$q(n-p+\beta)/p\not=n$,
the open set $G=\R^n\setminus E$ admits 
a $(q,p,\beta)$-Hardy--Sobolev inequality if
\[  n-p+\beta < \lambda_1
\qquad \text{or}\qquad 
\lambda_2<\frac{q}{p}(n-p+\beta) \,,\]
while if 
\[\lambda_1 < n-p+\beta \le \frac{q}{p}(n-p+\beta) \le \lambda_2\,, \]
then the $(q,p,\beta)$-Hardy--Sobolev inequality fails in $G$; notice that the
latter inequalities always hold for some parameters $q,p,\beta$ when
$0\le \lambda_1 < \lambda_2\le n$. Moreover, sets satisfying
the dimensional bounds in~\eqref{e.ex_main} do exist. An easy example would
be a closed, unbounded, non-porous, and countable set $E\subset \R^{m}$,
$m\le n-2$, which is
embedded to $\R^n$; then it is well-known that 
$\ldima(E)=\dimh(E)=0$ and $\udima(E)=m<n-1$. 
A concrete example of such a set, for $m=1$ and $n\ge 3$, 
is obtained by embedding a copy
of the set $E_0=\{1/j : j\in\N\}\cup\{0\}\subset [0,1]$ to each unit interval
$[k,k+1]$, $k\in\Z$, in the $x_1$ axis of $\R^n$.  
\end{example}

Notice that we do not know in Example~\ref{ex.basic} what
happens when $n-p+\beta = \lambda_1$. However, in the case $p=q$ it is known
that the $(p,\beta)$-Hardy inequality does not hold when $n-p+\beta = \dimh(G^c)$,
i.e., the lower bound in terms of the Hausdorff dimension is strict
in Theorem~\ref{t.dichotomy} for $p=q$,
see~\cite[Thm~1.1]{MR2442898}.
The proof of the strict inequality 
is heavily based on the known 
self-improvement of Hardy inequalities (cf.\ e.g.~\cite{MR1948106}). 
We do not know if such self-improvement 
holds in general for Hardy--Sobolev inequalities, 
but such a property would certainly
yield strict inequalities to the lower bounds of Theorem~\ref{t.dichotomy}
for all $p<q< p^*$, as well.

Nevertheless, let us also point out that the strict inequality
plays a much bigger role in the case $p=q$, since for instance when
the complement of $G$ is $\lambda$-regular, 
the strict inequality shows that the $(p,p-n+\lambda)$-Hardy inequality does not hold in $G$;
compare this to Corollary~\ref{coro.regular}.

\section{Sufficient conditions for global inequalities}\label{s.horiuchi}

For the next two sections, we change the perspective a little bit and
consider the validity of the global Hardy--Sobolev inequalities
for all functions $f\in C_0^\infty(\R^n)$ when the distance
is taken to a closed set $E\subset\R^n$ with $\lvert E\rvert =0$. 
In this section we establish the following sufficient condition
for such inequalities. Note that
Theorem~\ref{t.hs.thin} is an immediate consequence of 
the case $\udima(E)<n-1$ of  this more general result. 

\begin{thm}\label{t.hardy}
Let $E\not=\emptyset$ be a closed porous set in $\R^n$, and
let 
$1\le p\le q\le np/(n-p)<\infty$
and $\beta\in\R$ be such that
\[
\udima(E) < \frac{q}{p}(n-p+\beta)\,. 
\]
In addition, assume that either $\udima(E) < n-1$ or that
\begin{equation}\label{eq.extra}
\beta\le\frac{(p-1)(qp+np-nq)}{qp+p-q}\,.
\end{equation}
Then, there is a constant $C>0$ such that inequality 
\begin{equation}\label{e.weighted}
\biggl(\int_{\R^n} \lvert f(x)\rvert^q \, \delta_E(x)^{(q/p)(n-p+\beta)-n}\,dx\biggr)^{1/q}
\le C\biggl(\int_{\R^n} \lvert \nabla f(x)\rvert^p\,\delta_E(x)^\beta\,dx\biggr)^{1/p}
\end{equation}
holds for every $f\in C^\infty_0(\R^n)$.
\end{thm}

The proof of Theorem~\ref{t.hardy} follows from the results of 
Horiuchi~\cite{MR1021144} and
our characterization of the (upper) Assouad dimension, Theorem~\ref{t.hori_and_ass}.
Nevertheless, since the proof of inequality~\eqref{e.weighted} is somewhat implicit 
in~\cite{MR1021144}, where the interest is mainly in the corresponding
non-homogeneous inequalities, we have chosen to include here a detailed proof. 
This also allows us to clarify some rather subtle points in
the proof and to distinguish during the proof
when the full power of the $P(s)$-property is needed and when the 
claims follow more directly from the dimensional bound
$\udima(E) < \frac{q}{p}(n-p+\beta)$. 

We first reduce Theorem~\ref{t.hardy} 
to the following lemma, which is
essentially \cite[Proposition~5.1]{MR1021144}, corresponding to the
case $p=1$. The proof of Lemma~\ref{lem.Proposition_5.1} is
given at the end of this section after 
we have established an important technical tool
in Lemma~\ref{lem.third}. 
Throughout this section we say that a set $M$ in $\R^n$ is \emph{admissible} if $M$ is 
both open and bounded
and $\partial M$ is a closed $(n-1)$-dimensional manifold of class $C^\infty$.

\begin{lem}\label{lem.Proposition_5.1}
Let $E\not=\emptyset$ be a closed porous set in $\R^n$, and
let 
$1\le q\le n/(n-1)$
and $\beta\in\R$ be such that
$
\udima(E) < q(n-1+\beta)\,. 
$
If $n-1\le \udima(E) < n$, then we assume in addition that $\beta\le 0$.
Then there exists a
positive constant $C_1$ 
such that
\[
\bigg(\int_M \delta_E(x)^{q(n-1+\beta)-n}\, dx\bigg)^{1/q}\le C_1 
\int_{\partial M} \delta_E(x)^\beta d\mathcal{H}^{n-1}(x)\,,
\]
whenever $M\subset \R^n$ is an admissible set.
\end{lem}

Also the following result, 
which shows that
the Assouad dimension of $E$ is closely related to the
Muckenhoupt $A_1$-properties of the powers of the distance function,
will be needed in the 
proof of Theorem \ref{t.hardy}.

\begin{lem}\label{l.A_1}
Let $E\not=\emptyset$ be a closed set in $\R^n$ and let $\omega = \delta_E^{s-n}$,
where $\udima(E) < s\le n$. Then $\omega$ is a Muckenhoupt $A_1$-weight, i.e.,
there is a constant $c\ge 1$ such that inequality
\[
\vint_B \omega(x) \,dx \le c\,\mathop{\mathrm{ess\,inf}}_{x\in B}\,\omega(x)
\]
holds whenever $B$ is any ball in $\R^n$.  
\end{lem}

This result is not difficult to see from Lemma~\ref{l.aikawa} 
(see also \cite[Lemma 2.2]{MR1118940}),
and hence we omit the straightforward but somewhat tedious details.
For more
information on Muckenhoupt weights, we refer to \cite[Chapter IV]{MR807149}.

\begin{proof}[Proof of Theorem~\ref{t.hardy}]
Let us first consider the case $p=1$.
Fix a non-negative function $f\in C^\infty_0(\R^n)$.
By Sard's theorem \cite{MR0007523}, and the implicit function theorem,
the following two statements hold 
simultaneously for almost every $t>0$.
First, the  boundary of the 
compact
set 
\[ 
M_t=\{x\in \R^n\,:\,  f(x) >   t\}
\]
coincides with the level set $\{x\in \R^n\,:\,  f(x)=t\}$ and, second, this
level set is a compact $(n-1)$-dimensional 
manifold of class $C^\infty$.

By  Minkowski's integral inequality \cite[p. 271]{MR0290095}, 
the co-area formula (see e.g.~\cite[Theorem 3.2.12]{MR0257325} and Lemma~\ref{lem.Proposition_5.1},
we obtain
\begin{align*}
\bigg(\int_{\R^n} f(x)^{q}&\,\delta_E(x)^{q(n-1+\beta)-n}\,dx\bigg)^{1/q}
  \le \int_{0}^\infty \bigg(\int_{M_t}\delta_E(x)^{q(n-1+\beta)-n}\,dx\bigg)^{1/q}\,dt\\
& \le C_1  \int_0^\infty \int_{\partial M_t} \delta_E(x)^\beta d\mathcal{H}^{n-1}(x)\,dt
= C_1\int_{\R^n}  \lvert \nabla f(x)\rvert \delta_E(x)^\beta\,dx\,.
\end{align*}
(Note here that since $\udima(E) < q(n-1+\beta)\le n+\beta q$, both
$\delta_E^{q(n-1+\beta)-n}$ and $\delta_E^\beta$ are locally integrable
by Lemma~\ref{l.aikawa}(B).)
An approximation argument via mollification of $\lvert f\rvert$
concludes the proof of Theorem~\ref{t.hardy} when $p=1$;
the dominated convergence theorem (in the case of distance weights
with positive powers) or Lemma \ref{l.A_1}
and \cite[Theorem 2.1.4]{MR1774162}
(in case of negative powers) can be used to make the argument rigorous.

Let then $q,p$ and $\beta$ be as in the statement of Theorem~\ref{t.hardy}.
We assign 
\[\hat q = \frac{1}{1-1/p + 1/q}\,,\qquad \hat\beta = \frac{q(n-p+\beta)}{\hat q p} -n+1\,,
\qquad  \text{ and } \quad \hat f = \lvert f\rvert^{q/\hat q}\,.\] 
It is straightforward to verify that then $\hat q$ and $\hat\beta$ satisfy 
the  assumptions 
of the case $p=1$ of the
theorem (in particular, if $\beta$ satisfies the
additional bound~\eqref{eq.extra}, then $\hat\beta\le 0$).
Hence we obtain, using the case $p=1$ for $\hat f$ with 
 exponents  $\hat q$ and $\hat \beta$, that
\begin{equation}\label{e.reduce1}
\begin{split}
&\bigg( \int_{\R^n} \lvert f(x)\rvert^q \,\delta_E(x)^{(q/p)(n-p+\beta)-n}\,dx\bigg)^{1-1/p+1/q}
\\&= \bigg( \int_{\R^n} \lvert \hat f(x)\rvert^{\hat q} \,\delta_E(x)^{\hat q(n-1+\hat \beta)-n}\,dx\bigg)^{1/\hat q}
 \lesssim \int_{\R^n} \lvert \nabla \hat f(x) \rvert\,\delta_E(x)^{\hat \beta}\,dx\,;
\end{split}
\end{equation}
the standard mollification of $\hat f$ can be 
 used to justify  the last step
(see above). 
Since 
\[
\lvert \nabla \hat f(x) \rvert  = \lvert \nabla \lvert f\rvert^{q/\hat q}(x) \rvert
\le C(p,q) \lvert f(x)\rvert^{q(p-1)/p} \lvert \nabla f(x)\rvert
\]
almost everywhere, we can dominate the last integral in~\eqref{e.reduce1} 
with the help of H\"older's inequality by
\begin{equation}\label{e.reduce2}
\begin{split}
C(p,q) &\int_{\R^n} \lvert \nabla f(x)\rvert  \lvert f(x)\rvert^{q(p-1)/p}\,\delta_E(x)^{\hat\beta}\,dx\\
&\lesssim \bigg(\int_{\R^n} \lvert \nabla f(x)\rvert^{p}\,\delta_E(x)^{\beta}\,dx\bigg)^{1/p}
\bigg(\int_{\R^n} \lvert f(x)\rvert^{q}\,\delta_E(x)^{(q/p)(n-p+\beta)-n}\,dx\bigg)^{1-1/p}\,.
\end{split}
\end{equation}
The $(q,p,\beta)$-Hardy--Sobolev inequality 
 now  follows from 
estimates~\eqref{e.reduce1} and~\eqref{e.reduce2}.
\end{proof}

The proof of Lemma~\ref{lem.Proposition_5.1} is in turn based on the following
result (cf.~\cite[Proposition~6.3]{MR1021144}),
which utilizes a weighted Poincar\'e inequality in the case $\beta\le 0$
and a relative isoperimetric inequality when $\beta>0$.

\begin{lem}\label{lem.third}
Let $E\not=\emptyset$ be a closed porous set in $\R^n$, and
let $\beta\in\R$ be such that
$
\udima(E) < n+\beta\,; 
$
if $n-1\le \udima(E) < n$, then we assume in addition that $\beta\le 0$.
If
$M$ is an admissible set in $\R^n$ and $B$ is a ball in $\R^n$ such that
\begin{equation}\label{e.halfball}
\int_{M\cap B} \delta_E(x)^\beta \,dx = \tfrac 12 \int_B \delta_E(x)^\beta\,dx\,,
\end{equation}
then
\[
\int_{M\cap B} \delta_E(x)^\beta \,dx \le C \diam(B)\int_{\partial M\cap B} \delta_E(x)^\beta\,d\mathcal{H}^{n-1}(x)\,.
\]
Here the constant $C>0$ is independent of $M$ and $B$.
\end{lem}

\begin{proof}
We give an outline of the proof and, moreover, 
provide some additional details
that are not very explicit in~\cite{MR1021144}.

Let $B$ be ball a which satisfies the assumption~\eqref{e.halfball}.
We consider first  
the case $\beta \le 0$. 
The key tools in this case are:
(A) the inequality $\lvert M\cap B\rvert \le (1-\kappa)\lvert B\rvert$ 
with  a constant $\kappa >0$ 
independent of $M$ and $B$ 
(this inequality is not explicitly stated in~\cite{MR1021144} but is used there) and
(B) the weighted Poincar\'e inequality
\begin{equation}\label{e.poincare}
\int_B \lvert u(x)- u_B \rvert \,\delta_E(x)^\beta\,dx 
\le C\diam(B) \int_B \lvert \nabla u(x)\rvert \,\delta_E(x)^\beta\,dx
\end{equation}
that is valid for Lipschitz functions on $B$ (see e.g.~\cite[pp.~387--388]{MR1021144}).
We remark that the assumption 
$\udima(E)<n+\beta$ is used to establish both (A) and (B).
The key observation is that then
$\delta_E^\beta$ satisfies an $A_1$-condition
 by Lemma~\ref{l.A_1}, 
and so 
property (A) follows from \cite[Lemma~2.2 and Theorem~2.9]{MR807149}.

Having (A) and (B), the idea is to approximate the characteristic function of $M$ with
the Lipschitz functions
\[
u_\varepsilon(x) = \begin{cases} 1\,\qquad &x\in M^\varepsilon\\
\mathrm{dist}(x,\partial M)/\varepsilon\,,\qquad &x\in M\setminus M^\varepsilon\\
0\,,\qquad &x\in M^c\,,
\end{cases}
\]
where $M^\varepsilon = \{x\in M\,:\, \mathrm{dist}(x,\partial M) > \varepsilon\}$.
Indeed, by using (A), (B), and the coarea formula
\cite[Theorem~3.2.12]{MR0257325}, we obtain
\begin{align*}
 \int_{M\cap B} \delta_E(x)^\beta\,dx 
&=\lim_{\varepsilon \to 0+}\, \int_{M^\varepsilon\cap B} \delta_E(x)^\beta\,dx \le C\kappa^{-1} \diam(B) \, \lim_{\varepsilon \to 0+}\int_B \lvert \nabla u_\varepsilon(x)\rvert \,\delta_E(x)^\beta\,dx\\
&= C\kappa^{-1} \diam(B) \, \lim_{\varepsilon \to 0+} \frac{1}{\varepsilon} \int_0^\varepsilon \int_{B\cap \{x\in M\,:\, \mathrm{dist}(x,\partial M)=t\}} \delta_E(x)^\beta\, d\mathcal{H}^{n-1}(x)\,dt\\
&=C\kappa^{-1} \diam(B)\,\int_{\partial M \cap B} \delta_E(x)^\beta\, d\mathcal{H}^{n-1}(x)\,.
\end{align*}
Actually, in order to make the last limiting step rigorous, one
should first perform
the estimates above with the family of truncations 
\[\delta_E(x)^{\beta,\lambda} = \min \{\delta_E(x)^\beta,\lambda\}\,,\qquad \lambda>0\,,\]
as weights.
The resulting estimates turn out to be  uniform in $\lambda$; namely,
the truncations satisfy an $A_1$-condition with 
the  same constant as $\delta_E^\beta$ does. 
Consequently, the weighted
Poincar\'e inequality~\eqref{e.poincare} holds for these truncated weights
with a constant
independent of the truncation parameter $\lambda>0$,
and this yields the desired estimate also for the weight~$\delta_E^\beta$.

The case where $\beta>0$ and $\diam(B) < \dist(B,E)$ 
is quite straightforward,
since in this case we have 
$\delta_E(x)^\beta \simeq \dist(B,E)^\beta$ for every $x\in B$.
This equivalence, together with the assumption~\eqref{e.halfball},
used for both $M$ and $M^c$,
allows us to employ the isoperimetric
inequality~\cite[p.~163]{Mazya1985} as follows:
\begin{align*}
\int_{M\cap B} \delta_E(x)^\beta \,dx &\le \dist(B,E)^\beta \diam(B)
\min \{ \lvert M\cap B\rvert, \lvert M^c\cap B\rvert \}^{(n-1)/n}\\
&\lesssim \dist(B,E)^\beta \diam(B)\,
\mathcal{H}^{n-1}(\partial M\cap B)\\&\simeq \diam(B)\int_{\partial M\cap B}  \delta_E(x)^\beta\,d\mathcal{H}^{n-1}(x)\,.
\end{align*}

Finally, we consider 
the case where $\beta>0$ and $\diam(B) \ge \dist(B,E)$.
In this case we have the additional assumption that $\udima(E)<n-1$; 
this permits us to
fix a number $s>1$ such that $\udima(E)<n-s<n-1$. 
By Theorem~\ref{t.hori_and_ass}, 
the set $E$ then has the $P(s)$-property.
Set $\eta=\diam(B)/N$ for some large $N>1$ that is to be determined later,
and write \[
M_1=M\,,\qquad M_2=E_\eta^c=\{x\in \R^n\,:\,\delta_E(x)\ge \eta\}\,.\]
For simplicity, we first assume that $\partial M_2$ is a closed $(n-1)$-dimensional 
manifold of class $C^\infty$; the modifications required in the
general case are discussed at  the end of the proof.

Since $E$ is  a closed porous set and $\udima(E) < n+\beta$
(the latter is actually a triviality since $\beta>0$),
we have by Lemma~\ref{l.equiv} that 
\begin{equation}\label{e.starting}
\begin{split}
\diam(B)^n \sup_{x\in B} \delta_E(x)^\beta &\lesssim \diam(B)^{n+\beta}\simeq \int_B  \delta_E(x)^\beta\,dx\le 2 \int_{M\cap B} \delta_E(x)^\beta\,dx\\
&\le 2\sup_{x\in B} \delta_E(x)^\beta \,\lvert M_1\cap M_2\cap B\rvert + 2\int_{M_1\cap M_2^c\cap B} \delta_E(x)^\beta\,dx\,.
\end{split}
\end{equation}
Using the $P(s)$-property,
the last term in~\eqref{e.starting} can be estimated 
 as follows:
\begin{align*}
\int_{M_1\cap M_2^c\cap B} \delta_E(x)^\beta\,dx &\le \sup_{x\in B} \delta_E(x)^\beta\, \lvert E_\eta\cap B\rvert
\le CN^{-s}\sup_{x\in B} \delta_E(x)^\beta \, \diam(B)^{n}\,.
\end{align*}
From the $P(s)$-property, as in \cite[Proposition~6.1(4)]{MR1021144},
or from the estimate $\udima(E)<n-s$, as in \cite[Corollary~5.10]{KLV},
we obtain that
\[
 \mathcal{H}^{n-1} (\partial E_\eta \cap B)\le C\eta^{s-1} \diam(B)^{n-s} = CN^{1-s} \diam(B)^{n-1}\,,
\] 
where the constant $C$ is independent of both $B$ and $\eta$. 
Using inequalities~\eqref{eq.equiv} and~\eqref{e.halfball}, it is not hard to show that
\[
\lvert M_1\cap M_2 \cap B \rvert \le C \lvert (M_1\cap M_2)^c \cap B \rvert
\]
where $C$ is again independent $B$ and $\eta$.
An isoperimetric inequality \cite[p.~163]{Mazya1985} 
(we also refer to \cite[Lemma 4.5]{MR1021144}) 
then yields that
\begin{equation}\label{e.isoperimetrics}
\begin{split}
\lvert M_1\cap M_2 \cap B \rvert &\le  C\,d(B) \min\{ \lvert M_1\cap M_2 \cap B \rvert,
 \lvert (M_1\cap M_2)^c \cap B \rvert \}^{(n-1)/n}\\ 
&\le C\diam(B)\big( \mathcal{H}^{n-1}((M_1\cap \partial M_2)\cap B) 
+ \mathcal{H}^{n-1}((\partial M_1\cap M_2)\cap B)\big)\\
&\le C\diam(B)\bigg( \mathcal{H}^{n-1} (\partial E_\eta \cap B) + \eta^{-\beta} \int_{\partial M_1\cap B} \delta_E(x)^{\beta}\,d\mathcal{H}^{n-1}(x) \bigg)\\
&\le  CN^{1-s}  \diam(B)^{n} + CN^{\beta} \diam(B)^{1-\beta}\int_{\partial M\cap B} \delta_E(x)^{\beta}\,d\mathcal{H}^{n-1}(x)\,.
\end{split}
\end{equation}
Since $\sup\limits_{x\in B}\delta_E(x)^\beta \le C \diam(B)^\beta$ and $-s<1-s$, 
we conclude from inequalities above that
\begin{align*}
\diam(B)^n & \sup_{x\in B} \delta_E(x)^\beta \le CN^{1-s} \sup_{x\in B} 
\delta_E(x)^\beta  \diam(B)^n
+CN^{\beta}  \, \diam(B)\int_{\partial M\cap B} \delta_E(x)^{\beta}\,d\mathcal{H}^{n-1}(x)\,.
\end{align*}
As $s>1$, we can now choose $N$ to be so large that
$CN^{1-s}< \frac 1 2$, whence
the first term on the right-hand side 
can be absorbed to the left-hand side, 
and the claimed inequality follows as an easy consequence.

When $\partial M_2$ is not a closed $(n-1)$-dimensional $C^\infty$-manifold, 
the isoperimetric type inequality
(the second step in \eqref{e.isoperimetrics}) might fail. 
In this case, the following approximation is deployed. 
First, we fix a 
function $g\in C^\infty(\R^n\setminus \partial E_\eta)$  such that,
for each $x\in \R^n\setminus \partial E_\eta$,
\[
c_1\, \delta_{\partial E_\eta}(x) \le g(x) \le c_2\, \delta_{\partial E_\eta}(x)
\]
and $\lvert \nabla g(x)\rvert\le c_3$.
Here the constants $c_1$, $c_2$ and $c_3$ 
 can be chosen to be independent 
of $\partial E_\eta$, c.f.~\cite[VI.2.1]{MR0290095}.
 For $\varepsilon>0$, we write
\[
M_{2}^\varepsilon = M_{2} \cup \{x\in \R^n\,:\, g(x) \le \varepsilon \}\,.
\]
By applying the assumption $s>1$ and the $P(s)$-property, 
we obtain a sequence $(\varepsilon_j)_{j\in\N}$ of positive
non-critical values of $g$, converging to zero, such that each $G_j=\{x\,:\, g(x)=\varepsilon_j\}$ is 
an $(n-1)$-dimensional manifold of
class $C^\infty$ and 
\[\limsup_{j\to \infty} \mathcal{H}^{n-1}( G_j\cap B)\le C\eta^{s-1} d(B)^{n-s}\,,\]
where $C$ is independent of $B$ and $\eta$, we refer to \cite[p.~385]{MR1021144}. 
Since every $\varepsilon_j$ is a non-critical value of $g$,
 the set 
$\partial M_2^{\varepsilon_j}$ is an $(n-1)$-dimensional 
manifold of class $C^\infty$
and $\partial M_2^{\varepsilon_j}\subset \{x\in \R^n\,:\, g(x)=\varepsilon_j\}$. 
Indeed,
the boundary is locally represented by the latter level set which, in turn, is
given in terms of a non-critical value of $g$.
We can now adapt the estimates starting from~\eqref{e.starting}, 
 first  replacing $M_2$ by each 
$M_2^{\varepsilon_j}$, where $j\in\N$, and then
applying a limiting argument. \end{proof}

\begin{rem}
The last case 
in the proof of Lemma~\ref{lem.third}, 
where $\beta>0$ and $\diam(B) \ge \dist(B,E)$,
is the only instance
in the proof of Theorem~\ref{t.hardy} where
validity of the $P(s)$-property is needed
for all $0\le\eta_1<\eta_2$, in particular for
$\eta_1$ close to $\eta_2$. Moreover, here it
suffices that the $P(s)$-property holds for
some $s>1$. 
\end{rem}

To conclude the proof of Theorem~\ref{t.hardy}, we
prove Lemma~\ref{lem.Proposition_5.1} with the help of 
Lemma~\ref{lem.third}.

\begin{proof}[Proof of Lemma~\ref{lem.Proposition_5.1}]
Let us fix an admissible set $M$ in $\R^n$. 
First we notice that $\udima(E) < n+\beta$. Indeed, for $\beta\ge 0$ this is
trivial  by Lemma~\ref{l.aikawa}(D), 
 and for $\beta<0$ follows from the assumptions since
$0\le \udima(E) < q(n-1+\beta)\le n(n-1+\beta)/(n-1) < n + \beta$.
Thus $n+\beta\in\mathcal{A}(E)$ by Lemma~\ref{l.aikawa}(B), and
therefore $\delta_E^\beta$ is locally integrable.
When $x\in M$, the function
\[
\Lambda:[0,\infty)\to \R\,,\qquad \Lambda(r) = \int_{M\cap B(x,r)}\delta_E(y)^\beta\,dy 
- \tfrac 12 \int_{B(x,r)} \delta_E(y)^\beta\,dy\,,
\]
is continuous in $r$,  
$\Lambda(r)>0$ for small values of $r$, and
$\Lambda(r)\to -\infty$ as $r\to\infty$ by Lemma~\ref{l.equiv}. 
Hence, by the intermediate value theorem,
for each $x\in M$ there exists a ball $B(x,r(x))$ such that
\[
\int_{M\cap B(x,r(x))}\delta_E(y)^\beta\,dy = 
\tfrac 12 \int_{B(x,r(x))} \delta_E(y)^\beta\,dy\,;
\]
see also~\cite[Proposition~6.2]{MR1021144}.
By Besicovitch's covering theorem, see e.g.\ \cite[p.~30]{MR1333890}, 
we  thus  find a sequence of balls 
$\{B_j\}$ in $\{B(x,r(x))\,:\,x\in M\}$ 
such that $M$ is covered by the
union of the balls $\overline{B_j}$ and the overlap of these balls if uniformly bounded.

Using estimate~\eqref{eq.equiv} with 
 $s=q(n-1+\beta)>\udima(E)$, 
we obtain
\begin{equation}\label{e.start}
\begin{split}
\biggl(\int_M \delta_E(x)^{q(n-1+\beta)-n}\, dx\biggr)^{1/q}
&\le \sum_j \biggl(\int_{B_j\cap M} \delta_E(x)^{q(n-1+\beta)-n}\, dx\biggr)^{1/q}\\
&\lesssim \sum_j \diam(B_j)^{n/q} \bigl( \diam(B_j) + \dist(B_j,E) \bigr)^{n-1+\beta-n/q}\\
&\le \sum_j \diam(B_j)^{n/q} \bigl( \diam(B_j) + \dist(B_j,E)\big)^{\beta}
                       \diam(B_j)^{n-1-n/q}\\
& = \sum_j \diam(B_j)^{n-1} \bigl( \diam(B_j) + \dist(B_j,E)\bigr)^{\beta},
\end{split}
\end{equation}
where the penultimate step holds since $n-1-n/q\le 0$.
We continue the estimate in~\eqref{e.start}, 
first using~\eqref{eq.equiv} with $s=n+\beta>\udima(E)$,
and then Lemma~\ref{lem.third}, and conclude that
\begin{align*}
\sum_j d(B_j)^{n-1} & \bigl( d(B_j) + \dist(B_j,E)\bigr)^{\beta}
\lesssim \sum_j d(B_j)^{-1} \int_{B_j} \delta_E(x)^{\beta}\,dx\\
& \le 2\sum_j d(B_j)^{-1} \int_{M\cap B_j} \delta_E(x)^{\beta}\,dx\\
& \lesssim \sum_j \int_{\partial M\cap B_j} \delta_E(x)^{\beta}\,d\mathcal{H}^{n-1}(x)
\lesssim \int_{\partial M} \delta_E(x)^{\beta}\,d\mathcal{H}^{n-1}(x)\,.
\end{align*}
This proves Lemma~\ref{lem.Proposition_5.1}.
\end{proof}

\section{Necessary conditions for global inequalities}\label{s.application}

We turn to the necessary conditions for the global Hardy--Sobolev inequality~\eqref{e.weighted}.
The case $\beta\ge 0$ is straightforward, and in particular yields the necessity
part of Theorem~\ref{t.char}.

\begin{thm}\label{t.res}
Suppose that $E\not=\emptyset$ is a closed set in $\R^n$.
Let $1\le p,q  <\infty$ and $\beta\ge 0$ be such that $q(n-p+\beta)/p\not=n$ and that
inequality 
\begin{equation}\label{e.assumption}
\biggl(\int_{\R^n} \lvert f(x)\rvert^q \, \delta_E(x)^{(q/p)(n-p+\beta)-n}\,dx\biggr)^{1/q}
\le C\biggl(\int_{\R^n} \lvert \nabla f(x)\rvert^p\,\delta_E(x)^\beta\,dx\biggr)^{1/p}
\end{equation}
holds for every $f\in C^\infty_0(\R^n)$. 
Then \[\udima(E)< \frac{q}{p}(n-p+\beta)\,.\]
\end{thm}

\begin{proof}
Let $\varphi\in C^\infty_0(\R^n)$ be a function which is 
supported in the ball $B(0,2)$ and satisfies  $\varphi(y)=1$ if $y\in B(0,1)$.
Fix a point $x\in E$ and a radius $0<r<\diam(E)$. 
We write
\[
f(y)=\varphi((y-x)/r)
\]
for  every $y\in \R^n$. 
Since $\beta\ge 0$, we obtain from 
Lemma~\ref{l.aikawa}(B) that $\beta+n\in\mathcal{A}(E)$, 
and since $f(y)=1$ for all $y\in B(x,r)$ and $f$ is supported in $B(x,2r)$, 
it follows from Lemma~\ref{l.aikawa}(E) that 
\begin{align*}
\int_{B(x,r)} \delta_E(y)^{q(n-p+\beta)/p-n}\,dy 
& \le \int_{B(x,2r)} \lvert f(y)\rvert^q \, \delta_E(y)^{(q/p)(n-p+\beta)-n}\,dy
\\&\lesssim  \biggl(\int_{B(x,2r)} \lvert \nabla f(y)\rvert^p\,\delta_E(y)^\beta\,dy\biggr)^{q/p}\\
&\lesssim r^{-q}\biggl(\int_{B(x,2r)} \delta_E(y)^{\beta+n-n}\,\,dy\biggr)^{q/p}
\lesssim r^{q(n-p+\beta)/p}\,.
\end{align*}
This estimate shows that $q(n-p+\beta)/p>0$ and $q(n-p+\beta)/p\in \mathcal{A}(E)$, but 
then, by the self-improvement of Aikawa condition (see Lemma~\ref{l.aikawa}(C)), we have
$\udima(E)< q(n-p+\beta)/p$.
(Note that if $q(n-p+\beta)/p > n$ then we are
already done).
\end{proof}

The case $\beta<0$ is more technical, 
and based upon  the 
self-improvement of reverse H\"older inequalities.
In this context, we need to 
assume that $E$ is both porous and compact.
 We do not know if it is possible to remove  
the porosity assumption, but at least the
compactness assumption cannot be entirely omitted;
for further details, we refer to Example~\ref{e.noncompact}
and Remark~\ref{r.noncompact} below.

\begin{thm}\label{t.res.beta<0}
Suppose that $E\not=\emptyset$ is a compact and porous set in $\R^n$.
Let $\beta< 0$ and $1\le p<q<np/(n-p)<\infty$  be such that 
the Hardy--Sobolev inequality~\eqref{e.assumption}
 holds for every $f\in C^\infty_0(\R^n)$. Then
\[\udima(E)< \frac{q}{p}(n-p+\beta)\,.\]
\end{thm}

For the proof  of Theorem~\ref{t.res.beta<0}, 
we need the following result 
due to Iwaniec--Nolder \cite[Theorem~2]{MR802488},
which shows that the exponent 
on the right-hand side of the reverse 
H\"older inequality~\eqref{e.reverse} can
actually be improved to any $t>0$.

\begin{prop}\label{p.improvement}
Let $0<s<p$ and $f\in L^p_{\textup{loc}}(G)$, where $G\subset \R^n$ is an open set.
Suppose that for each cube $Q$ with $2Q\subset G$,
\begin{equation}\label{e.reverse}
\bigg(\vint_Q \lvert f(x)\rvert^p\,dx\bigg)^{1/p}\le A \bigg( \vint_{2Q} \lvert f(x)\rvert^s\,dx\bigg)^{1/s}\,,
\end{equation}
where the constant $A>0$ is independent of $Q$. Then for each
$t>0$, $\sigma>1$ and each cube $Q$ with $\sigma Q\subset G$, 
\[
\bigg(\vint_Q \lvert f(x)\rvert^p\,dx\bigg)^{1/p}\le C\bigg( \vint_{\sigma Q} \lvert f(x)\rvert^t\,dx\bigg)^{1/t}\,,
\]
where the constant $C>0$ depends only on $\sigma$, $n$, $p$, $s$, $t$ and $A$.
\end{prop}

\begin{proof}[Proof of Theorem \ref{t.res.beta<0}]
Since $E$ is porous, we may assume that $q(n-p+\beta)/p\not=n$. 
By the proof of Theorem~\ref{t.res}, it suffices to show
that $\beta + n\in\mathcal{A}(E)$, and 
by the assumptions and Lemma~\ref{l.aikawa}(A,B)
for this it suffices that $(q/p)\beta + n>0$ and $(q/p)\beta + n\in\mathcal{A}(E)$.  

To this end, let $Q$ be a cube in $\R^n$.
If $2Q\cap E\not=\emptyset$,   
we have for every $y\in Q$ that
\[
\delta_E(y)^{(q/p)\beta} =\delta_E(y)^{n+q-nq/p+(q/p)(n-p+\beta)-n}\le  
C\ell(Q)^{n+q-nq/p} \delta_E(y)^{(q/p)(n-p+\beta)-n}\,;
\]
notice that it follows from the assumptions that $n+q-nq/p\ge 0$.
Hence, inequality \eqref{e.assumption} with an appropriate
test function that is adapted to $Q$ shows that 
\begin{equation}\label{e.claim}\begin{split}
\bigg(\vint_Q\delta_E(y)^{(q/p)\beta}\,dy\bigg)^{p/q}
&\le \bigg(\ell(Q)^{n+q-nq/p} \vint_Q \delta_E(y)^{(q/p)(n-p+\beta)-n}\,dy\bigg)^{p/q} \\&\le C  
\vint_{2Q} \delta_E(y)^\beta\,dy\,.
\end{split}\end{equation}
On the other hand, if $2Q\cap E=\emptyset$, then $\delta_E(y)\simeq \dist(Q,E)$ 
for every $y\in Q$ 
which easily gives inequality~\eqref{e.claim}. 

Observe that $(q/p)(n-p+\beta)-n < \beta$;
indeed, if $n-p+\beta \le 0$ this is immediate, and
if $n-p+\beta > 0$, we obtain 
\begin{equation}\label{e.cont} 
\frac{q}{p}(n-p+\beta) < \frac{n}{n-p}(n-p+\beta) < n+\beta\,.  
\end{equation} 
Fix $x\in E$ and $R>0$ such that $E\subset B(x,R/2)$
(recall that $E$ was assumed to be compact).
By applying the assumed inequality~\eqref{e.assumption} 
to a function $f\in C^\infty_0(\R^n)$
which satisfies the condition $f(y)=1$ if $y\in B(x,R)$, 
we see that $\delta_E^\beta$ is locally integrable.
From this fact and inequality~\eqref{e.claim}
it follows that $\delta_E^\beta\in L^{q/p}_{\textup{loc}}(\R^n)$
and $(q/p)\beta + n > 0$. 
Choose $\varepsilon>0$ in such a way that $\udima(E)<\varepsilon\beta+n$. 
Then, since $0<1<q/p$, we may apply
the self-improvement of the reverse H\"older inequality~\eqref{e.claim},
 given by Proposition \ref{p.improvement}, 
to conclude  that inequality
\begin{equation}\label{e.claim2} 
\bigg(\vint_Q\delta_E(y)^{(q/p)\beta}\,dy\bigg)^{p/q} 
\le C \bigg(\vint_{2Q} \delta_E(y)^{\varepsilon\beta}\,dy\bigg)^{1/\varepsilon}
\end{equation}
holds for all cubes $Q$ in $\R^n$. Here $C$ depends on $n$, $p$, $q$, $\varepsilon$ and
the constant 
appearing in the inequality~\eqref{e.claim}.
But inequality~\eqref{e.claim2} and the fact that $\varepsilon\beta+n\in\mathcal{A}(E)$ clearly
imply that $(q/p)\beta + n\in\mathcal{A}(E)$ as 
was  required.
\end{proof}

To see that some  additional assumption is needed
for the set $E$ in Theorem~\ref{t.res.beta<0},
let us  consider the following example.

\begin{example}\label{e.noncompact}
Let $E$ be an $(n-k)$-dimensional subspace in $\R^n$, where $1\le k<n$;
then $E$ is a closed and porous set.
Let us fix numbers $1<p<q<np/(n-p)<\infty$ 
and $\beta=-k$. Then
\[
\ldima(E) = n-k > n-p+\beta\,.
\]
By Theorem~\ref{t.fractional_hardy_fat}, which we have postponed
to 
the following section,
the $(q,p,\beta)$-Hardy--Sobolev inequality
 \eqref{e.assumption} 
 actually holds
for all functions $f\in C^\infty_0(\R^n)$ 
 satisfying 
$f(x)=0$ whenever $x\in E$;
note that Corollary~\ref{coro.hardy_ass} is not enough to guarantee this.
On the other hand, if $f\in C^\infty_0(\R^n)$ and $f(x)\not=0$ for some
$x\in E$, then there is a point $y\in E$ such that $\nabla f(y)\not=0$. In particular,
\[
\int_{\R^n} \lvert \nabla f\rvert^p\,\delta_E^\beta\, dx = \infty\,,
\]
and consequently
the $(q,p,\beta)$-Hardy--Sobolev inequality 
\eqref{e.assumption}  holds trivially 
also for such  functions.
Thus  we conclude 
that the $(q,p,\beta)$-Hardy--Sobolev inequality \eqref{e.assumption}
holds for all $f\in C^\infty_0(\R^n)$, but nevertheless
$\udima(E) >  (q/p)(n-p+\beta)$;
indeed, if $n-p+\beta \le 0$ this is immediate, and
if $n-p+\beta > 0$, we can repeat the estimate \eqref{e.cont}
to see that
\[ \frac{q}{p}(n-p+\beta) < n+\beta = \udima(E)\,.  \] 
Thus we have shown that, contrary to the case $\beta\ge 0$,
the conclusion of Theorem~\ref{t.res.beta<0} 
does not hold for all closed (and porous) $E\subset\R^n$. 
\end{example}

\begin{rem}\label{r.noncompact}
While Example~\ref{e.noncompact} shows that
the compactness assumption in Theorem~\ref{t.res.beta<0} cannot be 
completely removed,
it can still be relaxed.
Indeed, the proof of the theorem reveals that we may replace
the assumption that $E$ is a compact set by the assumption that  
$E$ is a closed set such that, for every  $x_0\in E$,
\[
\inf \int_{\R^n} \lvert \nabla f\rvert^p \,\delta_E^\beta\,dx < \infty\,,
\]
where the infimum ranges over all $f\in C^\infty_0(\R^n)$ such that
$f(x_0)=1$. This weighted $p$-capacity condition is clearly satisfied by 
 all compact sets $E$.  
 On the other hand, when the set $E$ is non-compact, 
it clearly suffices to assume that  $\delta_E^\beta$ is
locally integrable,
which in turn follows, for instance, if we assume that $\udimm(E\cap B)<n+\beta$ for all balls
$B$ centered at $E$. Note that Example~\ref{e.noncompact} shows the sharpness of this
last condition.
\end{rem}

\section{The case of thick complements revisited}\label{s.thick}

We have the following
slight generalization for
the  `thick'  case of Corollary~\ref{coro.hardy_ass} when $\beta\le 0$
and $p<q<p^*$. 
While Theorem~\ref{t.fractional_hardy_fat} has also independent interest, the main reason 
why we have included it here is that  Example~\ref{e.noncompact} relies on 
this result.

\begin{thm}\label{t.fractional_hardy_fat} 
Let $\beta \le 0$ and $1<p<q<np/(n-p)<\infty$. 
Suppose that $G$ is a proper open set in $\R^n$, $n\geq 2$, such that $\ldima(G^c) > n-p+\beta$;
if $G$ is unbounded, we assume that $G^c$ is unbounded, as well.
Then there is a constant $C>0$ such that
\begin{equation}\label{e.full}
\bigg(\int_{G} \vert f(x)\vert^q\delta_{\partial G}(x)^{\frac{q}{p}(n-p+\beta)-n}\,dx\bigg)^{1/q}
\le C\biggl(
\int_{\R^n} \lvert \nabla f(x)\rvert^p \,\delta_{\partial G}(x)^{\beta}\, dx\biggr)^{1/p}\,
\end{equation}
whenever $f\in C^\infty_0(\R^n)$ and $f(x)=0$ for all $x\in G^c$.
\end{thm}

The proof of Theorem~\ref{t.fractional_hardy_fat} is based upon a general scheme, 
built in~\cite{MR3148524}, in combination with 
`pointwise Hardy' techniques, developed in~\cite{MR2854110,KoskelaLehrback2009,LeLip}.
The latter approach yields the following lemma,
which is a modification of the results in~\cite{LeLip}.

By $\mathcal{W}(G)$ we denote a Whitney decomposition 
(as in 
\cite[Section VI.1]{MR0290095})
of a proper open set $G$ in $\R^n$.
That is, the union of these dyadic 
cubes (whose interiors are pairwise disjoint) is the whole of $G$ and, moreover,
\begin{equation}\label{dist_est}
d(Q)\le \mathrm{dist}(Q,\partial G)\le 4\, d(Q)
\end{equation}
whenever $Q\in \mathcal{W}(G)$.

\begin{lem}\label{l.pointwise}
Let $\beta \le 0$ and $1<p<q<np/(n-p)<\infty$, 
and write $L=10\sqrt n$. 
Suppose that $G$ is a proper open set in $\R^n$, $n\geq 2$, 
such that $\ldima(G^c) > n-p+\beta$;
if $G$ is unbounded, we assume that $G^c$ is unbounded, as well.
Then there exists an exponent $1<r_0<p$ as follows:
For every $r_0<r<p$ there is a constant $C>0$
such that
\begin{equation}\label{e.pointwise}
\lvert f_{Q}\rvert^q \lesssim
\ell(Q)^{q - \frac{\beta q}{p}} \bigg(\vint_{LQ} \lvert \nabla f(x)\rvert^r
\,\delta_{\partial G}(x)^{\beta r/p}\,dx\bigg)^{q/r}
\end{equation}
whenever $Q\in\mathcal{W}(G)$ and $f\in C^\infty_0(\R^n)$ is such that 
$f(x)=0$ for all $x\in G^c$.
\end{lem}

\begin{proof}
The assumption $\ldima(G^c) > n-p+\beta$ implies that
there exists a positive number $\lambda > n-p+\beta$ such that
$\ell(Q)^{\lambda}\lesssim\Ha_\infty^\lambda(G^c\cap LQ)$
for all $Q\in\mathcal{W}(G)$; see e.g.~\cite[Remark~2.3]{KLV}
and notice that here we need to know the unboundedness
of $G^c$ if $G$ is unbounded.
On the other hand, by a simple modification of~\cite[Lemma~3.1(a)]{LeLip}
there exists $1<r_0<p$ such that
we have for all $r_0<r<p$ that
\begin{equation}\label{e.basic fat}
\Ha_\infty^\lambda(G^c\cap LQ)|f_Q|^r
\lesssim \ell(Q)^{r - \frac{\beta r}{p} - n + \lambda} \int_{LQ} 
\lvert \nabla f(x)\rvert^r\,\delta_{\partial G}(x)^{\beta r/p}\,dx,
\end{equation}
and actually  
the proof of~\cite[Lemma~3.1(a)]{LeLip} shows 
that for~\eqref{e.basic fat} it is enough to assume that
$f(x)=0$ for all $x\in G^c$.
Combining~\eqref{e.basic fat} with the above estimate 
$\ell(Q)^{\lambda}\lesssim\Ha_\infty^\lambda(G^c\cap LQ)$,
and taking everything to power $q/r$ yields the desired 
estimate~\eqref{e.pointwise}.
\end{proof}

\begin{proof}[Proof of Theorem \ref{t.fractional_hardy_fat}]
Write $L=10\sqrt n$, and let $1< r_0<p$ be 
as in Lemma~\ref{l.pointwise}. 
Fix a function $f\in C^\infty_0(\R^n)$ such that $f(x)=0$ for all $x\in G^c$.
By the covering property of Whitney cubes and inequality \eqref{dist_est}, \begin{equation}\label{e.split}
\int_{G}
\vert f(x)\vert^q \, \delta_{\partial G}(x)^{\frac{q}{p}(n-p+\beta)-n}\,dx
\lesssim\sum _{Q\in \mathcal{W}(G)}
\ell (Q)^{\frac{q}{p}(n-p+\beta)}
\bigg\{\vint_{Q} |f(x)-f_{Q}|^q\,dx +
|f_{Q}|^q\bigg\}\,.
\end{equation}
We  choose $r_0<r<p$ 
such that  $1/r-1/q<1/n$. Then 
the integral on the right-hand side of~\eqref{e.split}
can be estimated using a
$(q,r)$-Poincar\'e inequality for cubes:
\[\begin{split}
\vint_{Q} \lvert f(x)-f_{Q}\rvert^q\,dx & \lesssim \ell(Q)^{q} \bigg(\vint_{Q} \lvert \nabla f(x)\rvert^r
\,dx\bigg)^{q/r}\,\\
& \lesssim
\ell(Q)^{q - \frac{\beta q}{p}} \bigg(\vint_{LQ} \lvert \nabla f(x)\rvert^r
\,\delta_{\partial G}(x)^{\beta r/p}\,dx\bigg)^{q/r}.
\end{split}\]
On the other hand, a corresponding estimate for the second 
integral on the right-hand side of~\eqref{e.split}
follows from Lemma~\ref{l.pointwise} since $r_0<r<p$.
That is,
\[
\lvert f_{Q}\rvert^q \lesssim
\ell(Q)^{q - \frac{\beta q}{p}} \bigg(\vint_{LQ} \lvert \nabla f(x)\rvert^r
\,\delta_{\partial G}(x)^{\beta r/p}\,dx\bigg)^{q/r}.
\] 
Insertion of these two estimates to~\eqref{e.split} yields that
\begin{align}\label{upper_bound_beta}
\int_{G}
\vert f(x)\vert^q \,& \delta_{\partial G}(x)^{\frac{q}{p}(n-p+\beta)-n}\,dx\nonumber \\
&\lesssim
\sum _{Q\in \mathcal{W}(G)} \ell (Q)^{\frac{q}{p}(n-p+\beta)}
\ell(Q)^{q - \frac{\beta q}{p}} \bigg(\vint_{LQ} \lvert \nabla f(x)\rvert^r
\,\delta_{\partial G}(x)^{\beta r/p}\,dx\bigg)^{q/r} \nonumber\\
&\lesssim
\sum _{Q\in \mathcal{W}(G)}
|Q|^{q/p} \bigg(  \vint_{LQ}  
\lvert \nabla f(x)\rvert^r\,\delta_{\partial G}(x)^{\beta r/p} \,dx\bigg)^{q/r}\,.
\end{align}

Since $1<r<p<q<\infty$, we have 
for every $g\in L^{p/r}(\R^n)$ that 
\begin{equation}\label{beta_ineq}
\begin{split}
\sum_{Q\in\mathcal{W} (G)} |Q|^{q/p}\bigg(\vint_{LQ} \lvert g(x)\rvert \,dx\bigg)^{q/r}
&\lesssim \bigg(\int_{\R^n}  |g(x)|^{p/r}\,dx\bigg)^{q/p}\,.
\end{split}
\end{equation} 
Indeed, to obtain inequality~\eqref{beta_ineq},  one first dominates
the left-hand side of \eqref{beta_ineq} by (a constant multiple of)
\[
\sum_{Q\in\mathcal{W}(G)} \int_{\R^n}\chi_Q(x) \big(I_\sigma \lvert g\rvert(x)\big)^{q/r}\,dx
\]
and then applies the pairwise disjointedness of the interiors of Whitney cubes and 
the boundedness 
of the Riesz potential
$I_\sigma \colon  h\mapsto \lvert x\rvert^{\sigma-n}*h$, where $\sigma = nr(q/p-1)/q\in (0,n)$,
from $L^{p/r}(\R^n)$ to $L^{q/r}(\R^n)$; we refer to \cite[Theorem~1]{MR0312232}. 

Now estimates~\eqref{upper_bound_beta}
and~\eqref{beta_ineq} (the latter with $g=|\nabla f|^r\delta_{\partial G}^{\beta r/p}$ ; 
if $g\notin L^{p/r}(\R^n)$ the claim is trivial)
show that, indeed 
\begin{align*}
\int_{G}
\vert f(x)\vert^q \, \delta_{\partial G}(x)^{\frac{q}{p}(n-p+\beta)-n}\,dx
&\lesssim
\sum _{Q\in \mathcal{W}(G)}
|Q|^{q/p} \bigg(  \vint_{LQ}  
\lvert \nabla f(x)\rvert^r\,\delta_{\partial G}(x)^{\beta r/p} \,dx\bigg)^{q/r}\\
&\lesssim
\bigg(\int_{\R^n} 
\lvert \nabla f(x)\rvert^p\,\delta_{\partial G}(x)^{\beta} \,dx\bigg)^{q/p}\,,
\end{align*}
and inequality~\eqref{e.full} follows.
\end{proof}

\section{Dimensional dichotomy for the complement}\label{s.nec}

In this final section, we establish dimensional dichotomy results for 
the complements of open sets which admit Hardy--Sobolev inequalities.
In particular, the global dichotomy result, stated in Theorem~\ref{t.dichotomy}, 
is proved at the end of this section.
Before that, we study local versions of these results 
in the following Propositions~\ref{p.dichotomy} and~\ref{p.dichotomybeta<0}
corresponding to the cases $\beta\ge 0$ and
$\beta<0$, respectively.
Similar results for the
unweighted $p$-Hardy inequality were proven in~\cite{MR1948106}, and 
for the weighted $(p,\beta)$-Hardy inequality in~\cite{MR2442898}.
Recall that in the `Hardy' case $q=p$, both of the dimensional bounds
for the complement $G^c$ are strict, and that we do not know if this is
true 
for the lower bounds
also when $q>p$ (see the discussion after 
Example~\ref{ex.basic}).

\begin{prop}\label{p.dichotomy}
 Let $G\subset\R^n$ be an open set and
 assume that $1\le p\le q<np/(n-p)<\infty$ and $\beta \ge 0$ 
 are such that $q(n-p+\beta)/p\not=n$ and $G$ admits a $(q,p,\beta)$-Hardy--Sobolev 
 inequality~\eqref{e.hardy-sobo}.
 Then for each (closed) ball $B=\overline{B}(x,R)\subset\R^n$ either
 \[
 \udima(G^c\cap B)< \frac q p (n-p+\beta)\quad \text{or}\quad 
\dimh(G^c\cap 2B)\ge n-p+\beta\,.
\]
\end{prop}

\begin{proof} 
The proof is based on the approach in~\cite{MR1948106}, and
uses also ideas from the proofs of Lemmas~6.4 and~6.6 in~\cite{lehrbackHardyAssouad}.
Fix $x_0\in\R^n$ and $R>0$, and write $B_0=\overline{B}(x_0,R)\subset\R^n$. 
It suffices to show that if $\dimh(G^c\cap 2B_0) < n-p + \beta$, then 
$\udima(G^c\cap B_0 )< (q/p)(n-p+ \beta)$, and without
loss of generality we may also assume that $0<q(n-p+\beta)/p<n$.

If, indeed, $\dimh(G^c\cap 2B_0) < n-p + \beta$, then
$\Ha^{n-p+\beta}_\infty(G^c\cap 2B_0) = 0$. 
Thus, assuming that $f\in  C_0^\infty(G\cup B(x_0,2R))$ is fixed 
(and $f\not \equiv 0$),
we can find 
balls $B_i^j=B(w_{ij},r_{ij})$ with $w_{ij}\in G^c\cap 2B_0$, $i=1,\dots,N_j$, so that
$G^c \cap 2B_0 \subset\bigcup_{i=1}^{N_j} B_i^j$, 
and 
\[
\sum_{i=1}^{N_j} r_{ij}^{n-p+\beta} \leq \|f\|_\infty^{-p}\, 2^{-j}
\] 
for all $j\in\N$.
Define cut-off functions $\psi_j(y)=\min_i\{1,r_{ij}^{-1}\delta_{2B_i^j}(y)\}$
and set $f_j = \psi_j f$. 
Then 
$f_j$ is clearly a Lipschitz function with a compact support in $G$
and 
\[|\nabla f_j|^p \lesssim \sum_i r_{ij}^{-p} \chi_{3B_i^j}|f|^p+|\nabla f|^p\]
almost everywhere. 
Moreover, we have that 
$\lim_{j\to\infty} f_j= f$ in $G$. 

Since
$1\le p,q<\infty$ and 
 $\delta_{\partial G}$ is bounded and away from zero
in the support of $f_j$, the $(q,p,\beta)$-Hardy--Sobolev inequality holds 
also for the
function $f_j$ by the standard approximation, 
and this with the choice of the balls $B_i^j$  implies that
\[\begin{split}
\biggl(\int_{G} |f_j|^q \delta_{\partial G}^{(q/p)(n - p + \beta) - n}\,dx\biggr)^{p/q} &
     \le C  \int_G |\nabla f_j|^p\delta_{\partial G}^\beta\,dx\\
  &  \le  C \biggl\{ 
      \|f\|^p_\infty \sum_{i=1}^{N_j}  |B_i^j|   r_{ij}^{-p+\beta} + 
               \int_G |\nabla f|^p\delta_{\partial G}^\beta\,dx
                \biggr\} \\
  &  \le  C 2^{-j} + C \int_{G} |\nabla f|^p\delta_{\partial G}^\beta\,dx\,.
  \end{split}
\]
By Fatou's lemma
this argument shows that the $(q,p,\beta)$-Hardy--Sobolev 
inequality~\eqref{e.hardy-sobo}
holds for all functions
$f\in C_0^\infty(G\cup B(x_0,2R))$ with a constant 
$C=C(C_1,q,p,\beta,n)>0$,
where $C_1>0$ is the constant for which the $(q,p,\beta)$-Hardy--Sobolev
inequality was assumed 
to hold for all $f\in C_0^\infty(G)$.

Let then  $w\in G^c\cap B_0$ and $0<r<R/2$.
By the above reasoning, 
we can now use the $(q,p,\beta)$-Hardy--Sobolev
inequality for the function $f(y)=\varphi((y-w)/r)$, where
$\varphi$ is as in the proof of Theorem~\ref{t.res}; indeed,
$f\in C_0^\infty(G\cup B(x_0,2R))$. 
Since 
$\lvert G^c \cap 2B_0\rvert =0$, a calculation 
similar to the one given in the 
proof of Theorem~\ref{t.res} shows that 
\begin{equation}\label{e.test u}
\begin{split}
 \int_{B(w,r)} \delta_{G^c}(x)^{(q/p)(n-p+\beta)-n}\,dx= \int_{B(w,r)} 
 \delta_{\partial G}(x)^{(q/p)(n-p+\beta)-n}\,dx 
\le C r^{(q/p)(n-p+\beta)}.
\end{split}
\end{equation}
Here the constant $C$ is independent of both $w$ and $r$.
Using Lemma~\ref{l.aikawa}(E) and
the fact that
$\delta_{G^c}\le \delta_{G^c\cap B_0}$, we infer
from~\eqref{e.test u} that $(q/p)(n-p+\beta)\in\mathcal{A}(G^c\cap B_0)$,
and thus Lemma~\ref{l.aikawa}(C) yields the claim $\udima(G^c\cap B_0) < (q/p)(n-p+\beta)$. 
\end{proof}

For $\beta<0$, 
the case $q>p$ involves additional difficulties 
compared to the case $q=p$ that was considered 
in~\cite{MR2442898}.

\begin{prop}\label{p.dichotomybeta<0}
Let $G\subset\R^n$ be an open set. Assume that
 $1\le p< q<np/(n-p)<\infty$  and $\beta<0$
are such that  
 $G$ admits a $(q,p,\beta)$-Hardy--Sobolev 
 inequality~\eqref{e.hardy-sobo}.
Then, by writing $\lambda = 8\sqrt n$, we have either
\[
\udima(G^c\cap B)< \frac q p(n-p+\beta)\quad \text{or} \quad
\ldimm(G^c\cap \lambda B)\ge n-p+\beta
\]
whenever  $B=\overline{B}(x,R)\subset\R^n$  is a ball
such that $\delta_{G^c}^\beta \in L^1(\lambda B)$ and 
 $G^c\cap \lambda B$ is porous. 
\end{prop}

\begin{proof}
The proof follows the lines of the proof of Proposition~\ref{p.dichotomy},
but many of the details are different  and 
hence we include here a complete proof.
Fix a ball  $B_0=\overline{B}(x_0,R)$ such that $\delta_{G^c}^\beta \in L^1(\lambda B_0)$ and
$\udima(G^c\cap \lambda B_0)<n$, i.e.\ 
 $G^c\cap \lambda B_0$ is porous; recall Lemma \ref{l.aikawa}(D). 
It suffices to show  
 $\udima(G^c\cap B_0)< (q/p)(n-p+ \beta)$ while assuming that $\ldimm(G^c\cap \lambda B_0)< n-p+\beta$. 
 Without loss of generality, we may also assume that 
$0<q(n-p+\beta)/p<n$. 

Fix  $f\in C^\infty_0(G\cup B(x_0, 6\sqrt nR))$ 
 such that  $f\not\equiv 0$.
Since 
$\ldimm(G^c\cap \lambda B_0)< n-p+\beta$,
there is a sequence $(r_j)_{j\in\N}$ of positive numbers, converging to zero and satisfying the following
two conditions for each $j\in\N$: (i) $r_j\le 2\sqrt nR$ and $r_j\le \delta_{G^c}(y)$ whenever 
$y\in G\cap (B(x_0,6\sqrt nR))^c$
is such that $f(y)\not=0$, and (ii)
there are balls $B_i^j=B(w_{ij},r_j)$ with $w_{ij}\in G^c\cap \lambda B_0$, $i=1,\dots,N_j$, so that
$G^c \cap \lambda B_0 \subset\bigcup_{i=1}^{N_j} B_i^j$,
and 
\[
N_jr_j^{n-p+\beta} \le  2^{-j}\,\lVert f\rVert_\infty^{-p}\,.
\]
Define cut-off functions $\psi_j(y)=\min_i\{1, r_j^{-1} \delta_{2B_i^j}(y)\}$
and let $f_j=\psi_j f$ for each $j\in \N$.
Then $f_j$ is a Lipschitz function that is compactly supported in $G$.
Moreover, a careful inspection shows that
\begin{equation}\label{e.newes}
\lvert \nabla f_j\rvert^p \lesssim \sum_i r_j^{-p} \chi_{3B_i^j}\lvert f\rvert^p 
 \chi_{\{r_j\le \delta_{G^c}\}} + \lvert \nabla f\rvert^p
\end{equation}
almost everywhere, 
and $\lim_{j\to\infty} f_j= f$ in $G$. 
Let us remark that 
compared to the proof of Proposition~\ref{p.dichotomy},
the new factor $\chi_{\{r_j\le \delta_{G^c}\}}$ appears. 
 This will be needed in the subsequent
arguments due to the
assumption that $\beta<0$, 
and this is the reason why 
the upper bound is now given in terms of 
the lower Minkowski dimension instead of the Hausdorff dimension.

Using approximation and  Fatou's lemma, and 
applying the assumed
$(q,p,\beta)$-Hardy--Sobolev inequality and inequality~\eqref{e.newes}, we obtain
\[\begin{split}
\biggl(\int_{G} |f|^q \delta_{\partial G}^{(q/p)(n - p + \beta) - n}\,dx\biggr)^{p/q} 
&\le \liminf_{j\to\infty} 
\biggl(\int_{G} |f_j|^q \delta_{\partial G}^{(q/p)(n - p + \beta) - n}\,dx\biggr)^{p/q} 
  \\& \le \lim_{j\to\infty} C\bigg(   2^{-j} +  \int_{G} |\nabla f|^p\delta_{\partial G}^\beta\,dx\bigg)
= C \int_{G} |\nabla f|^p\delta_{\partial G}^\beta\,dx\,.
  \end{split}
\]
Hence the $(q,p,\beta)$-Hardy--Sobolev inequality~\eqref{e.hardy-sobo}
holds in fact  for all functions
$f\in C_0^\infty(G\cup B(x_0, 6\sqrt nR))$  with a constant $C=C(C_1,q,p,\beta,n)>0$,
where $C_1>0$ is the constant  for which the $(q,p,\beta)$-Hardy--Sobolev
inequality holds for all $f\in C_0^\infty(G)$.

Arguing as in the proof of Theorem~\ref{t.res.beta<0}
and using the fact that $\lvert G^c \cap \lambda B_0\rvert=0$, we obtain
a constant $\kappa>0$ such that
\begin{equation}\label{e.weak}
\bigg(\vint_Q\delta_{G^c}(y)^{(q/p)\beta}\,dy\bigg)^{p/q} \le \kappa \vint_{2Q} 
\delta_{G^c}(y)^{\beta}\,dy
\end{equation}
for all cubes $Q$ such that $2Q\subset B(x_0, 6\sqrt nR)$. 
However, we will actually need
an improved version of~\eqref{e.weak},
as in the proof of Theorem~\ref{t.res.beta<0}.
Let  $Q_0$ be an open cube centered at $x_0$ and with side length $4R$. Then 
$2Q_0\subset B(x_0,6\sqrt nR)$.
Choose $\varepsilon>0$ for which $\udima(G^c\cap \lambda B_0)<\varepsilon\beta+n.$ 
Let us observe that
$1<q/p$ and $\delta_{G^c}^\beta\in L^{q/p}(Q_0)$ by inequality~\eqref{e.weak} and the assumptions. Hence, 
by Proposition~\ref{p.improvement} there exists a constant $C>0$ such that
the left hand side of \eqref{e.weak} is dominated 
by $C\bigl(\vint_{2Q} \delta_{G^c}^{\varepsilon \beta} \bigr)^{1/\varepsilon}$ 
if $2Q\subset Q_0$.

With the help of 
the $(q,p,\beta)$-Hardy--Sobolev inequality for
$C_0^\infty(G\cup B(x_0,6\sqrt nR))$ and the improved version
of inequality~\eqref{e.weak}, we can now 
proceed as follows.
Fix a cube $Q$ that
is centered at $G^c\cap B_0$ and whose
side length is bounded by $\tfrac{2}{5} R$; then
$4Q\subset Q_0$ and $\delta_{G^c}(y)=\delta_{G^c\cap \lambda B_0}(y)$ whenever $y\in 4Q$. Hence,
\begin{align*}
\bigg(\int_Q \delta_{G^c}(y)^{(q/p)(n-p+\beta)-n}\,dy\bigg)^{p/q}
&\le C\ell(Q)^{n-p} \vint_{2Q} \delta_{G^c}(y)^{\beta}\,dy\\&
\le \ell(Q)^{n-p}\bigg(\vint_{2Q}\delta_{G^c}(y)^{(q/p)\beta}\,dy\bigg)^{p/q} 
\\&\le C \ell(Q)^{n-p}\bigg(\vint_{4Q} \delta_{G^c}(y)^{\varepsilon\beta}\,dy\bigg)^{1/\varepsilon}\\
&= C \ell(Q)^{n-p}\bigg(\vint_{4Q} \delta_{G^c\cap \lambda B_0}(y)^{\varepsilon\beta}\,dy\bigg)^{1/\varepsilon}
\le C\ell(Q)^{n-p+\beta}\,.
\end{align*}
Since $\delta_{G^c}\le \delta_{G^c \cap B_0}$ we can again use
Lemma~\ref{l.aikawa}(E,C) to conclude that 
\[\udima(G^c\cap B_0)<\frac q p(n-p+\beta)\,.\qedhere \]
\end{proof}

The global dichotomy results of Theorem~\ref{t.dichotomy} 
can now be proved using similar arguments as in the
local results  of Propositions~\ref{p.dichotomy}
and~\ref{p.dichotomybeta<0}. We outline the main ideas:

\begin{proof}[Proof of Theorem~\ref{t.dichotomy}]
Let us first consider the case $\beta \ge 0$. 
It suffices to prove that inequality
\begin{equation}\label{e.target}
\udima(G^c) < \frac q p (n-p+\beta)
\end{equation}
holds if $\dimh(G^c)<n-p+\beta$ and $q(n-p+\beta)/p<n$.
Fix $\omega \in G^c$ and $0<r<\diam(G^c)$, and write $B_0=\overline{B}(\omega,3r)$. Since
$\dimh(G^c\cap 2B_0)\le \dimh(G^c)<n-p+\beta$, we can repeat
the proof of Proposition \ref{p.dichotomy}
to obtain that 
\[
\int_{B(\omega,r)} \delta_{G^c}(x)^{(q/p)(n-p+\beta)-n}\,dx \le C r^{(q/p)(n-p+\beta)}\,,
\]
where the constant $C$ is independent of both $\omega$ and $r$. Thus, 
$(q/p)(n-p+\beta)\in \mathcal{A}(G^c)$ and inequality~\eqref{e.target} follows
from Lemma~\ref{l.aikawa}(C).

In the case $\beta < 0$,
the claim for $q=p$ follows
from the results in~\cite{MR2442898}, and hence
it suffices to prove inequality~\eqref{e.target}
while assuming that $p<q$ and $\ldimm(G^c)<n-p+\beta$.
In particular, then $\lvert G^c\rvert = 0$. Arguing as in the proof of 
Proposition~\ref{p.dichotomybeta<0},
we find that 
the $(q,p,\beta)$-Hardy--Sobolev inequality  \eqref{e.assumption} 
 actually holds for all $f\in C^\infty_0(\R^n)$.
In particular, the
assumptions of Theorem~\ref{t.res.beta<0} 
are valid with $E=G^c$, and so inequality~\eqref{e.target} follows.
\end{proof}

\bibliographystyle{abbrv}

 \def\cprime{$'$}

\end{document}